\documentclass[12pt,a4paper]{article}

\usepackage{amssymb}
\usepackage{amsthm}
\usepackage{amsmath}

\usepackage{color}
\usepackage[active]{srcltx} 
\usepackage{graphicx}	
\usepackage[english]{babel}

\newtheorem{theorem}{Theorem}
\newtheorem{proposition}[theorem]{Proposition}
\newtheorem{lemma}[theorem]{Lemma}
\newtheorem{corollary}[theorem]{Corollary}
\newtheorem{remark}[theorem]{Remark}
\newtheorem*{remark*}{Remark}

\newtheorem{definition}[theorem]{Definition}

\title{Integral characterization for Poincar\'e half-maps in planar linear systems}

\author{Victoriano Carmona\thanks{
Escuela Polit\'ecnica Superior.  Calle Virgen de \'Africa, 7. 41011 Sevilla. Spain. Email: vcarmona@us.es}
\and  Fernando Fern\'andez-S\'anchez\thanks{
Escuela T\'ecnica Superior de Ingenier\'ia. Camino de los Descubrimientos s/n. 41092 Sevilla. Spain. Email: fefesan@us.es}}

\date{
Dpto. Matem\'atica Aplicada II \& IMUS. \\ Universidad de Sevilla.\\
}                    

\begin{document}

\maketitle

\begin{abstract}
The intrinsic nature of a problem usually suggests a first suitable method to deal with it.
Unfortunately, the apparent ease of application of these initial approaches may make their
possible flaws seem to be inherent to the problem and often no alternative ways to solve it are searched for.
For instance,
since linear systems of differential equations are easy to integrate, Poincar\'e half-maps 
for piecewise linear systems are always studied
by using the direct integration of the system in each zone of linearity. However, this approach is
accompanied by two important defects: due to the different 
spectra
of the involved matrices, many cases and strategies must be considered  and,
since the flight time appears as a new variable,
nonlinear complicated equations arise.  

This manuscript is devoted to present a novel theory to characterize Poincar\'e half-maps in planar linear systems, that avoids the computation of their solutions and the problems it causes. 
This new perspective rests on the use of line integrals 
of a specific conservative vector field which is orthogonal to the flow of the linear system. Besides the obvious mathematical interest, this approach is attractive because it allows to simplify the study of 
piecewise-linear systems and deal with open problems in this field.
\end{abstract}

\textbf{Keywords:} Piecewise planar linear systems, Poincar\'e half-maps, Inverse integrating factors.

\textbf{MSC2010}: 34A26, 34A36, 34C05.

\tableofcontents

\section{Introduction}\label{sec:intro}
Linearization, Lyapunov stability, normal forms, and index theory are some fundamental tools for the analysis of dynamical systems. Another one, the construction of Poincar\'e maps, is specially suitable for the study of existence, uniqueness, and stability of periodic orbits, homoclinic connections, and heteroclinic cycles. 

In the case of piecewise systems, the analysis of orbits that cross the separation manifolds between different regions leads naturally to the use of Poincar\'e maps
that are therefore defined as composition of transition maps, sometimes called the Poincar\'e half-maps, between the separation boundaries. Usually, the explicit calculation of these maps is a difficult task because it depends obviously on the equations of the involved systems. However, in the case of piecewise linear systems, direct integration of the equations may be used in each region to obtain feasible expressions. 

Unfortunately, this advantage of linear systems, that is, the possibility of performing direct integration of the equations, has also two important weaknesses for the construction of Poincar\'e half-maps. The first one is that the computation of the solutions of the linear systems, together with its subsequent study, is strongly conditioned by the spectrum of the matrix of the system and the final expression of Poincar\'e half-maps is written in terms of the eigenvalues.
This fact forces the appearance of many different cases to study. The second weak point is the inevitable (non-linear) dependence of the Poincar\'e half-maps on the flight time, namely, the time spent by the orbit between two consecutive intersection points with the separation manifolds.

These two deficiencies are even further important if we take into account that almost every works about Poincar\'e maps for piecewise linear systems use direct integration of the systems, what is accompanied by large case-by-case studies. The valuable works \cite{AnViKh66,FrPoRoTo98,LlTe14,MeTo15,LlPoVa19}, sorted by year of publication, are a few examples of these case-by-case studies from the early years to nowadays. Moreover, each one of these different cases requires individual techniques. This fact hinders and slows the research on the dynamic behavior of piecewise linear systems.

The main motivation of this work is how to override the flaws in the analysis of the transition maps due to performing the integration of planar linear systems. The obvious procedure is to avoid the computation of these integrals. 
In order to do it, we develop a new technique to characterize Poincar\'e half-maps of linear systems in a common, unique, and more suitable expression, without annoying
exhaustive divisions into cases 
and without the unnecessary dependence on the flight time. This new approach is the main goal of this manuscript.

Beyond the importance on its own of the new characterization of Poincar\'e half-maps, the relevance of this approach is also made evident by the simplification of the study of many important issues related to planar piecewise linear systems; for instance, the analyticity 
of Poincar\'e half-maps at tangency points (that is given in this work) or the open problem about providing optimal upper bounds on the number of limit cycles (see \cite{CaFeNoPR1,CaFeNoPR2}).

Loosely speaking, the characterization of the Poincar\'e half-map
related to the Poincar\'e section $\Sigma\equiv\{x=0\}$
 for a generic planar linear system in Lienard form
$$
\left\{\begin{array}{rcl}
\dot{x}&=&Tx-y,
\\ \dot{y} &=& Dx-a
\end{array}\right.
$$
where $a$, $T$, $D$ are real numbers, says that
the following assertions are equivalent:
\begin{enumerate}
\item[a)] $y_1$ is the image of $y_0$ by means of a Poincar\'e half-map,
\item[b)] $y_0,y_1\in\mathbb{R}$ satisfy $y_0 \, y_1\leqslant 0$ and
$$
  \operatorname{PV}\int_{y_1}^{y_0} \frac{-y}{D y^2-a T y+a^2} dy =c T, 
$$
for three concrete values of constant $c\in\mathbb{R}$. Here, $\operatorname{PV}$ stands for the
Cauchy Principal Value defined at \eqref{eq:PV}.
\end{enumerate}
Notice that the Cauchy Principal Value is only necessary for the case $a=0$, where the integral is improper and divergent due to a singularity at the origin.
In the interest of rigor, conditions for the existence of the Poincar\'e half-map and the integral must be added to the domains of variables $y_0$, $y_1$, and to the values of the parameters. Moreover, the three values of $c$ depend on the parameters and the relative location of the Poincar\'e section and the equilibrium of the system, if it exists.

The formal statement of the main result needs the definition of concepts, development of ideas, 
and establishment of preliminary results that are sequentially presented in this manuscript for the sake of better understanding.
Its proof is a direct consequence of the reciprocal results given in Theorem \ref{th:implicaderecha} and Theorem \ref{th:implicaizquierda}. Furthermore, in this second theorem, an integral expression for the flight time is also given.

The basic idea behind our approach to characterize the Poincar\'e half-maps is very easy and it can be sketched as follows (the details are given in the next sections). Divide the linear vector field that defines the system by a suitable function (namely, an inverse integrating factor) to get an orthogonal conservative vector field. Choose any $y_0, y_1\in\mathbb{R}$ such that  $(0,y_0)$ and $(0,y_1)$ are connected by a piece of an orbit of the system, $y_0 y_1\leqslant0$, and the closed curve formed by this piece of orbit and the segment joining $(0,y_0)$ and $(0,y_1)$ is a Jordan curve. The integral of the orthogonal vector field along this Jordan curve may take just three values that do not depend on $y_0$ and $y_1$. They just depend on the relative position between the Jordan curve and the equilibrium point of the system. Following these simple steps, a nice and manageable integral implicit equation for the Poincar\'e half-map has been obtained.

Naturally, this idea can be also extended to non-linear systems. As we have said in the previous paragraph, the only requirement is the existence of a suitable inverse integrating factor that allows the obtention of a reasonable integral implicit equation.

Therefore, this work is divided into several parts which are devoted to present the new concepts and ideas needed to prove the reciprocal results given in Theorem \ref{th:implicaderecha} (in Section \ref{sec:integral}) and Theorem \ref{th:implicaizquierda} (in Section \ref{sec:recovery}). The first logical step for our analysis is to give  an accurate definition and a detailed description of Poincar\'e half-maps associated to a straight line for planar linear systems. The location of the equilibrium of the linear system, if it exists, relative to the straight line allows to classify all the possible Poincar\'e half-maps into just three scenarios. This is done in Section \ref{sec:poincare}.

Now, the target of Section \ref{sec:integral} is the construction of an alternative way to write the
existing relationship between a point $y_0$ and its image $y_1$ by the
Poincar\'e half-map. More specifically, the beginning of this section is a brief summary of definitions and results about inverse integrating factors, a basic tool to obtain the suitable vector field that is integrated on appropriate closed curves to reach the desired alternative expression. This integral expression brings together all possible geometric configurations into a unique common function, in variables $y_0$ and $y_1$, one of whose level curves is the graph of the Poincar\'e half-map. Moreover, in this section, it is put into evidence a natural relationship between this common function and the index of a closed curve. This suggests for the function the name of the \emph{index-like function}.

Section \ref{sec:index-like} corresponds to the study of the main properties of the index-like function and its set of level curves. The primary results of this section concern, on the one hand, the analyticity and bijectivity of the implicit functions defined by the level curves and, on the other hand, the existence of a third-order differential system whose orbits are the graphs of these implicit functions. All the properties obtained in Section \ref{sec:index-like} are used in Section \ref{sec:recovery} to prove Theorem \ref{th:implicaizquierda} and thus to close the characterization of the Poincar\'e half-maps (and to get the flight time) in terms of the index-like function.
Finally, some conclusions and future works are given in Section \ref{sec:conclusions}.

\section{Poincare half-maps for planar linear systems}\label{sec:poincare}

Let us consider, for $\mathbf{x}=(x_1,x_2)^T$, the autonomous linear system 
\begin{equation}
\label{eq:sislinnohom}
    \dot{\mathbf{x}}=M \, \mathbf{x}+\mathbf{b}
\end{equation}
where $M=(m_{ij})_{i,j=1,2}$ is a real matrix and $\mathbf{b}=(b_1,b_2)^T\in\mathbb{R}^2$. Let us chose the Poincar\'e section $\Sigma\equiv\{x_1=0\}$.

Although the further analysis can be performed directly to system \eqref{eq:sislinnohom}, it is a good idea to reduce previously the number of parameters.
Note that if coefficient $m_{12}$ vanishes, system (\ref{eq:sislinnohom}) is uncoupled in such a way that a Poincar\'e half-map on section $\Sigma$ can not be defined (no return is possible).
Therefore, from now on, let us assume that $m_{12}\ne0$, what is usually called the observability condition \cite{ieee}. Under this assumption, the linear change of variable $x=x_1$, 
$
y=m_{22} x_1-m_{12} x_2-b_1,
$
 allows to write system \eqref{eq:sislinnohom}  into the generalized Lienard form,
 \begin{equation}\label{eq:lienard}
\left(\begin{array}{l} \dot{x}\\ \dot{y} \end{array}\right)=\left(\begin{array}{rr} T & -1\\
D& 0\end{array}\right) \left(\begin{array}{l} x\\ y \end{array}\right)-\left(\begin{array}{l}0\\a\end{array}\right),
\end{equation}
 where $a=m_{12}b_2-m_{22}b_1$ and $T$ and $D$ stand for the trace and the determinant of matrix $M$ respectively. Let us call $A$ the matrix of system \eqref{eq:lienard} and
\begin{equation}
\label{eq:lienard_L}
L(x,y)=(Tx-y,Dx-a)
\end{equation}
the corresponding vector field. In the new coordinates, since $x_1=x$, the Poincar\'e section $\Sigma$ remains the same.

The first equation of system \eqref{eq:lienard} evaluated on section $\Sigma=\{x=0\}$ is reduced to $\dot{x}|_{\Sigma}=-y$. Therefore, the flow of the system
crosses $\Sigma$ from the half-plane $\{x>0\}$ to $\{x<0\}$ when $y>0$, from the half-plane $\{x<0\}$ to $\{x>0\}$ when 
$y<0$ and it is tangent to $\Sigma$ at the origin.

The goal of this section is to define, in the usual way, the Poincar\'e half-maps of system (\ref{eq:sislinnohom}) corresponding to the section $\Sigma$ and to show that, in spite of the distinct cases that may appear in terms of the spectrum of $A$, there are only three different geometric types of Poincar\'e half-maps. 

Without loss of generality, since system \eqref{eq:lienard} is invariant under the change $(x,y,a)\longleftrightarrow(-x,-y,-a)$, it is only necessary to define the left Poincar\'e half-map. That is, let us consider $(0,y_0)\in\Sigma$ with $y_0\geqslant0$ and let be
\begin{equation}
\label{eq:orbitaflujo}
\Psi(t;y_0)=(\Psi_1(t;y_0),\Psi_2(t;y_0))
\end{equation}
the orbit of system \eqref{eq:lienard} that satisfies $\Psi(0;y_0)=(0,y_0)$. If there exists a value $\tau(y_0)>0$ such that $\Psi_1(\tau(y_0);y_0)=0$ and $\Psi_1(t;y_0)<0$ for every $t\in(0,\tau(y_0))$, we say 
that $y_1=\Psi_2(\tau(y_0);y_0)\leqslant0$ is the image of $y_0$ by the left Poincar\'e half-map, denoted by $y_1=P(y_0)$, and the value $\tau(y_0)$ is the corresponding left flight time. See Fig.~\ref{fig:poincare}(a).

\begin{remark}\label{rem:tangenciainvisible}
In the case that $P(0)$ can not be defined in this previous way but for every $\varepsilon>0$ there exist $y_0\in(0,\varepsilon)$
and $y_1\in(-\varepsilon,0)$
such that $P(y_0)=y_1$, the left Poincar\'e half-map can be extended with $P(0)=0$. This case corresponds to an equilibrium at the origin or a scenario known as an invisible tangency \cite{KuRiGr03} for half-plane $\{x<0\}$. See Fig.~\ref{fig:poincare}(b).  

Notice that in this case the left flight time can be also extended to the origin. For an invisible tangency, since $\tau(y_0)$ tends to $0$ as $y_0$ tends to $0$, the left flight time $\tau(0)$ should vanish. However, when the origin is an equilibrium, since the existence of the left Poincar\'e half-map implies that $4D-T^2>0$ (i.e., the equilibrium point is a center or a focus) and it is known that $\tau(y_0)=\frac{2\pi}{\sqrt{4D-T^2}}$ for any $y_0>0$ (see \cite{FrPoRoTo98}) then the natural choice is $\tau(0)=\frac{2\pi}{\sqrt{4D-T^2}}$.
\end{remark}

\begin{figure}[!h]
    \begin{center}
    \begin{tabular}{cc}
     \includegraphics[width=0.4\linewidth]{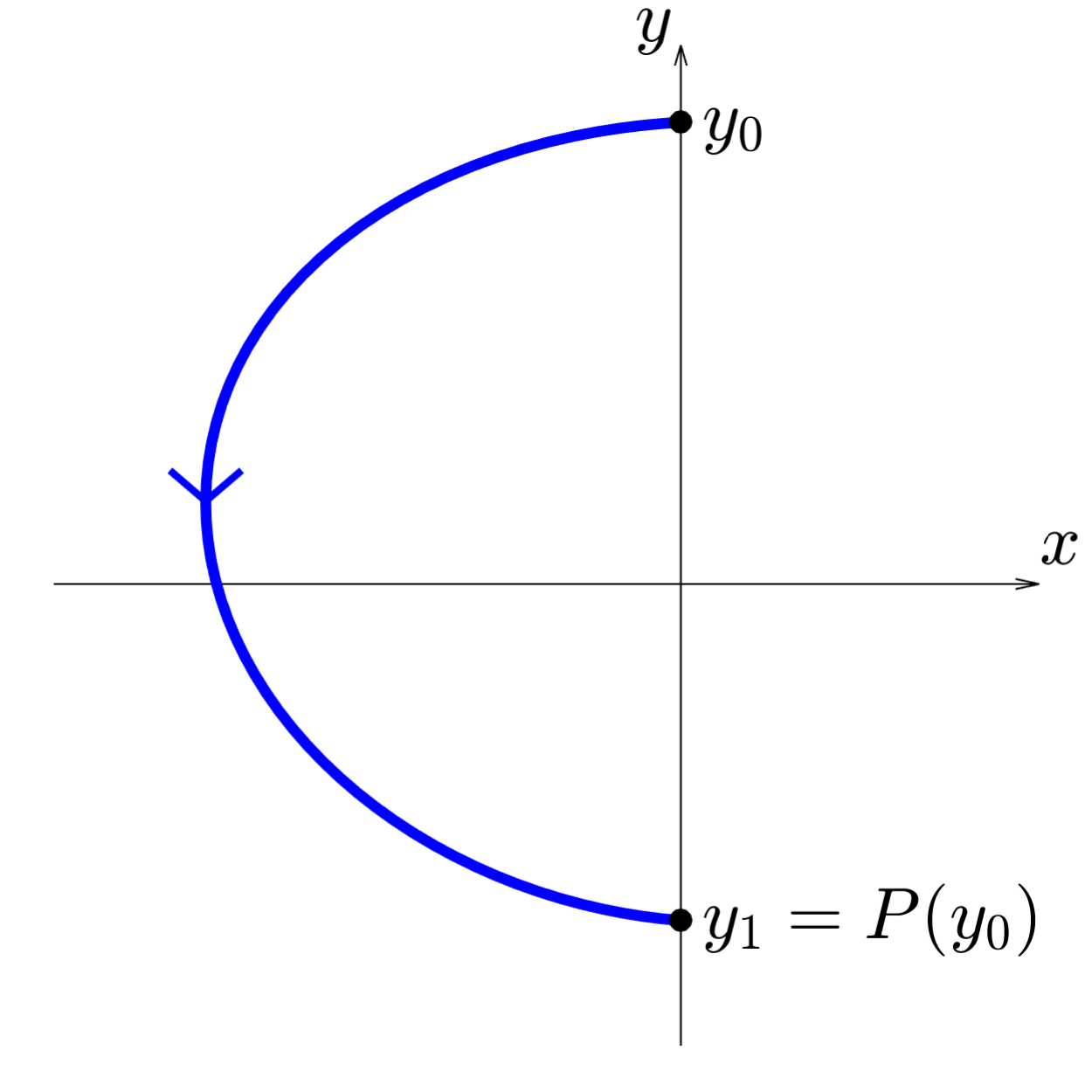}& \hspace{0.5cm}
     \includegraphics[width=0.4\linewidth]{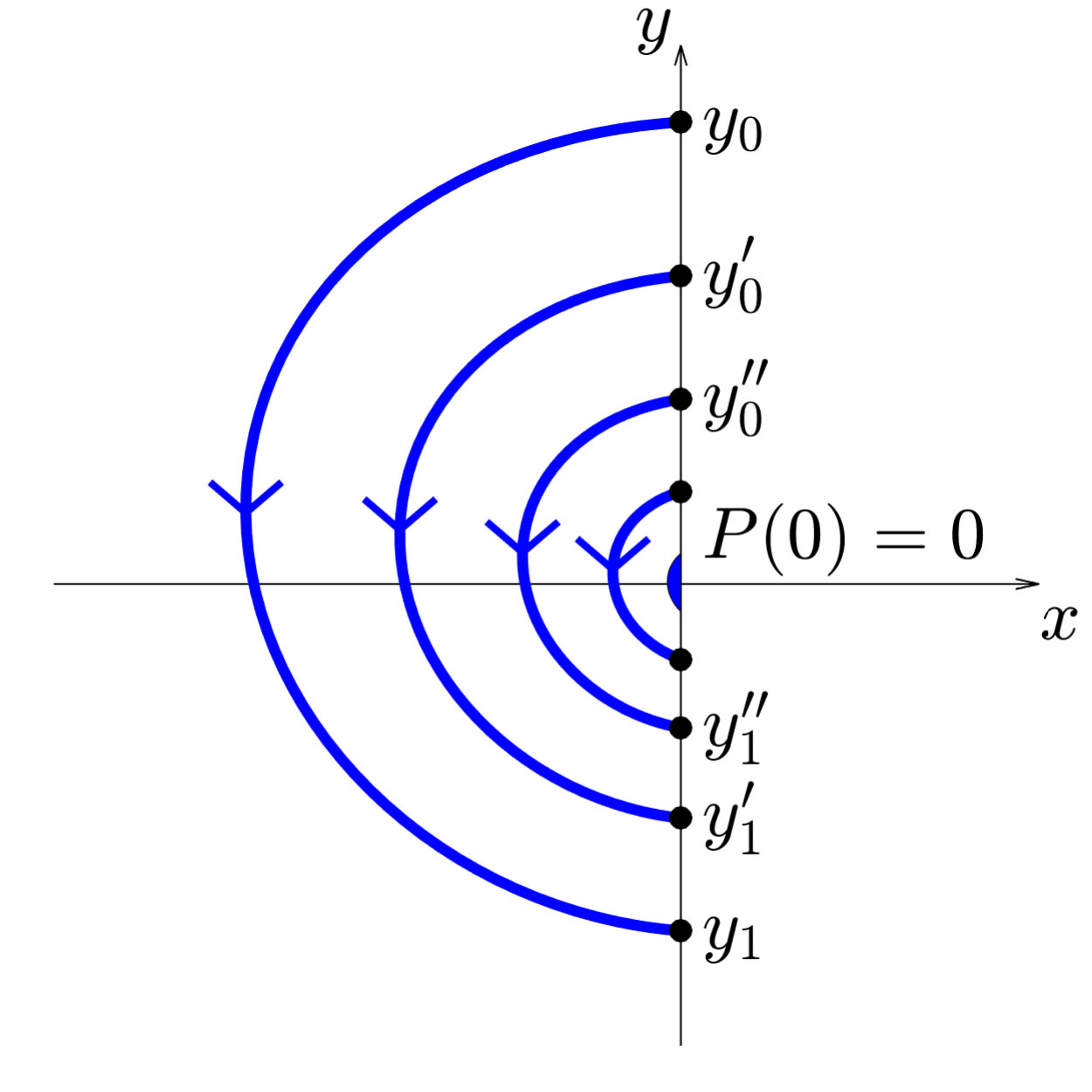}\\ 
     {\footnotesize (a)} & \hspace{0.5cm} {\footnotesize (b)}
     \end{tabular}
    \end{center}
     \caption{(a) Schematic drawing of the construction of the left Poincar\'e half-map. (b) Left Poincar\'e half-map for an invisible tangency or an equilibrium point at the origin.}\label{fig:poincare}
\end{figure}

From now until the end of this section, we assume that the following hypothesis holds: 
\begin{enumerate}
\item[(H)] \emph{There exist $y_0\geqslant0$  and $\tau(y_0)>0$ such that 
$\Psi_1(\tau(y_0);y_0)=0$ and $\Psi_1(t;y_0)<0$ for every $t\in(0,\tau(y_0))$.} 
\end{enumerate}
Let us consider the value $y_1=P(y_0)$, the segment $\Gamma_1=\{(0,y)\in\mathbb{R}^2: y\in(y_1,y_0)\}$, the piece of orbit $\Gamma_2=\{\Psi(t;y_0): t\in[0,\tau(y_0)]\}$ and the Jordan curve 
\begin{equation}
\label{eq:Jordancurve}
\Gamma=\Gamma_1\cup\Gamma_2.
\end{equation}
Note that the set $\Gamma_1$ may be the empty set (when $y_1=y_0=0$).
Under assumption (H), linear system \eqref{eq:lienard} has, at most, one equilibrium point. Note that if the system had infinitely many equilibria (that is, the equality $D^2+a^2=0$ holds) the straight line $y=T x$ would be foliated by these equilibrium points. This fact contradicts the theorem of existence and uniqueness of solutions since the straight line $y=T x$ would intersect the piece of orbit $\Gamma_2$.

Depending on the relative position of $\Gamma$ and the equilibrium point of system \eqref{eq:lienard}, if any, the following three mutually exclusive scenarios appear.
\begin{definition}
\label{def:escenarios}
Let us name the different scenarios as:
\begin{enumerate}
\item[(S$_2$)] The equilibrium point exists and it belongs to the interior of Jordan curve $\Gamma$, $\operatorname{Int}(\Gamma)$.
\item[(S$_1$)] The equilibrium point is the origin or, equivalently, the equilibrium exists and it belongs to segment $\Gamma_1$.
\item[(S$_0$)] All other cases, i.e., the set $\Gamma\cup\operatorname{Int}(\Gamma)$ contains no equilibrium points of system \eqref{eq:lienard}.
\end{enumerate}
\end{definition}
The reason for the numbering of the different scenarios comes from the values $k$ that will be given in Theorem \ref{th:implicaderecha}.

Let us briefly describe 
the three scenarios. In scenario (S$_2$), the equilibrium of system \eqref{eq:lienard} is located at the half-plane $\{x<0\}$ and the only possible configurations for the phase portrait are a center, a stable focus, and an unstable focus (that is, $4D-T^2>0$ and $a<0$). In the center case the domain of definition $\mathcal{D}$ of the left Poincar\'e half-map is the interval $[0,+\infty)$ and its range $\mathcal{R}$ is the interval $(-\infty,0]$. In the stable focus case, there exist a value $\hat{y}_0>0$ such that $P(\hat{y}_0)=0$, the domain is $\mathcal{D}=[\hat{y}_0,+\infty)$ and the range is $\mathcal{R}=(-\infty,0]$. For the unstable focus case, there is a value $\hat{y}_1<0$ such that $P(0)=\hat{y}_1$, the domain is $\mathcal{D}=[0,+\infty)$ and the range is $\mathcal{R}=(-\infty,\hat{y}_1]$. 
See Fig.~\ref{fig:casoS2}.

\begin{figure}[!h]
    \begin{center}
    \begin{tabular}{c@{}c@{}c}
     \includegraphics[width=0.33\linewidth]{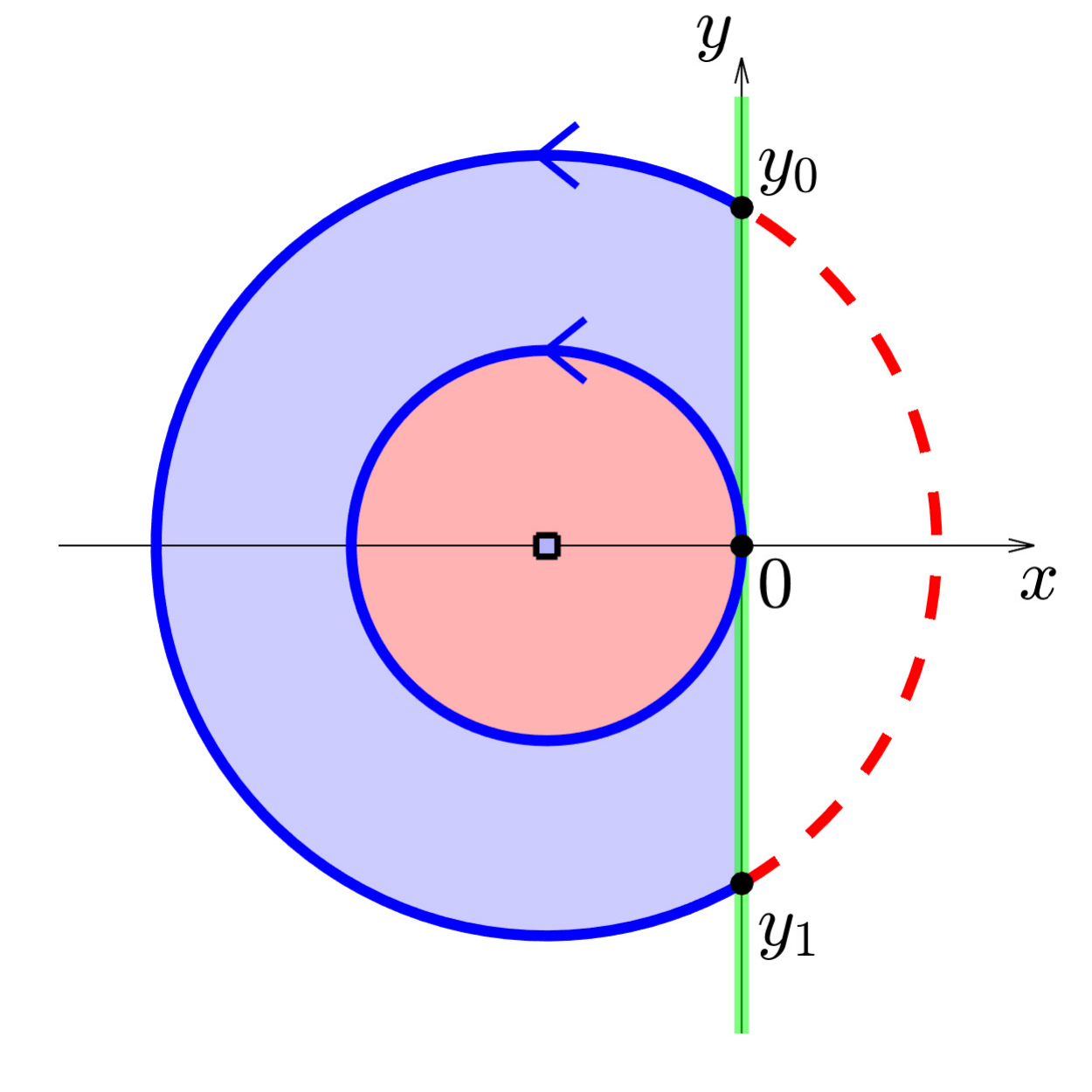}& 
     \includegraphics[width=0.33\linewidth]{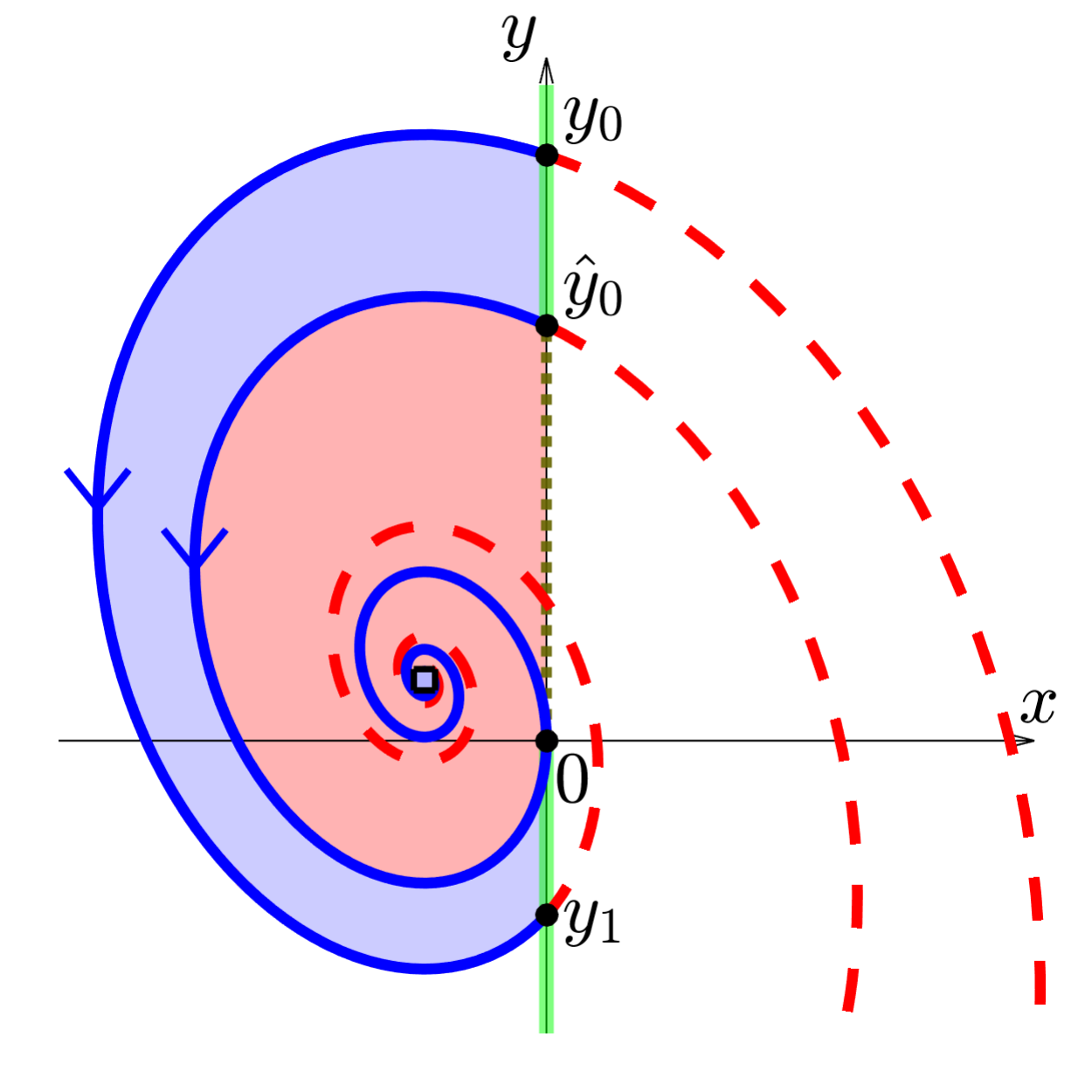}& 
     \includegraphics[width=0.33\linewidth]{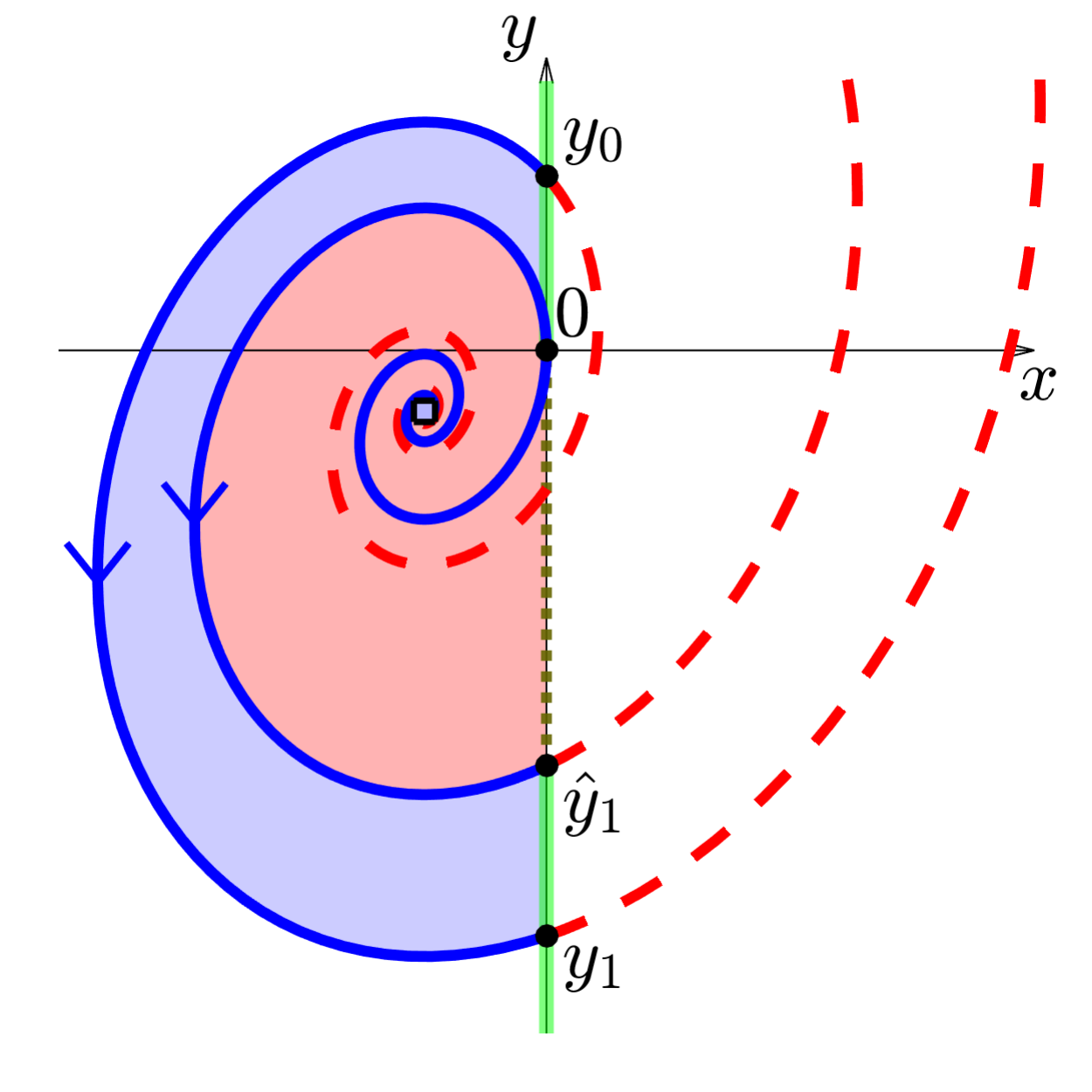}\\ 
     {\footnotesize (a)} & {\footnotesize (b)}& {\footnotesize (c)}
     \end{tabular}
    \end{center}
     \caption{Scenario (S$_2$): (a) center, (b) stable focus, (c) unstable focus.}\label{fig:casoS2}
\end{figure}

For scenario (S$_1$), the equilibrium of system \eqref{eq:lienard} is located at the origin and the only possible configurations for the phase portrait are a center, a stable focus, and an unstable focus (that is, $4D-T^2>0$ and $a=0$). In all these cases, the left Poincar\'e half-map can be extended to the origin by means of the definition $P(0)=0$ (see Remark \ref{rem:tangenciainvisible}). Therefore the domain is $[0,+\infty)$ and the range is $(-\infty,0]$.
See Fig.~\ref{fig:casoS1}.

\begin{figure}[!h]
    \begin{center}
    \begin{tabular}{c@{}c@{}c}
     \includegraphics[width=0.33\linewidth]{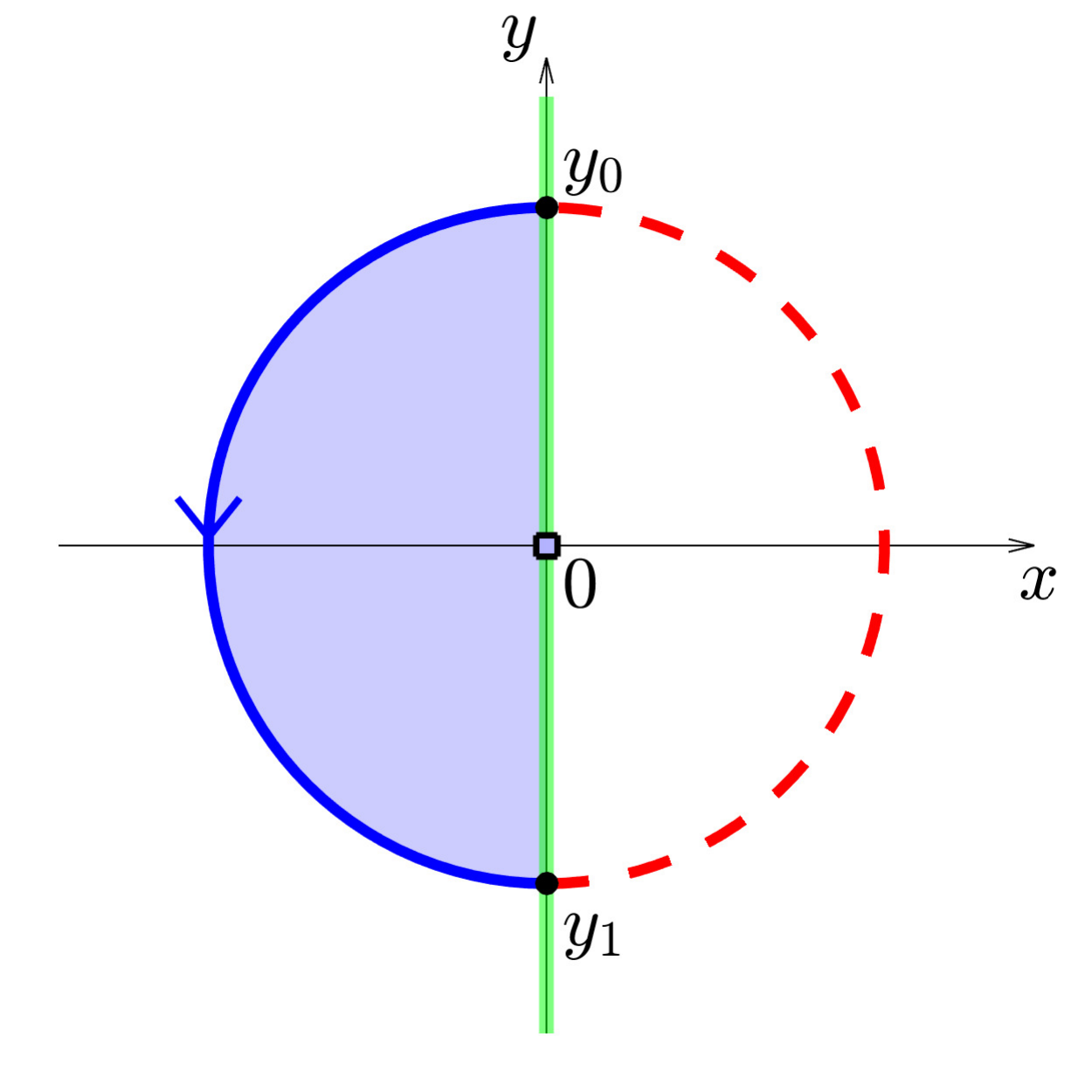}& 
     \includegraphics[width=0.33\linewidth]{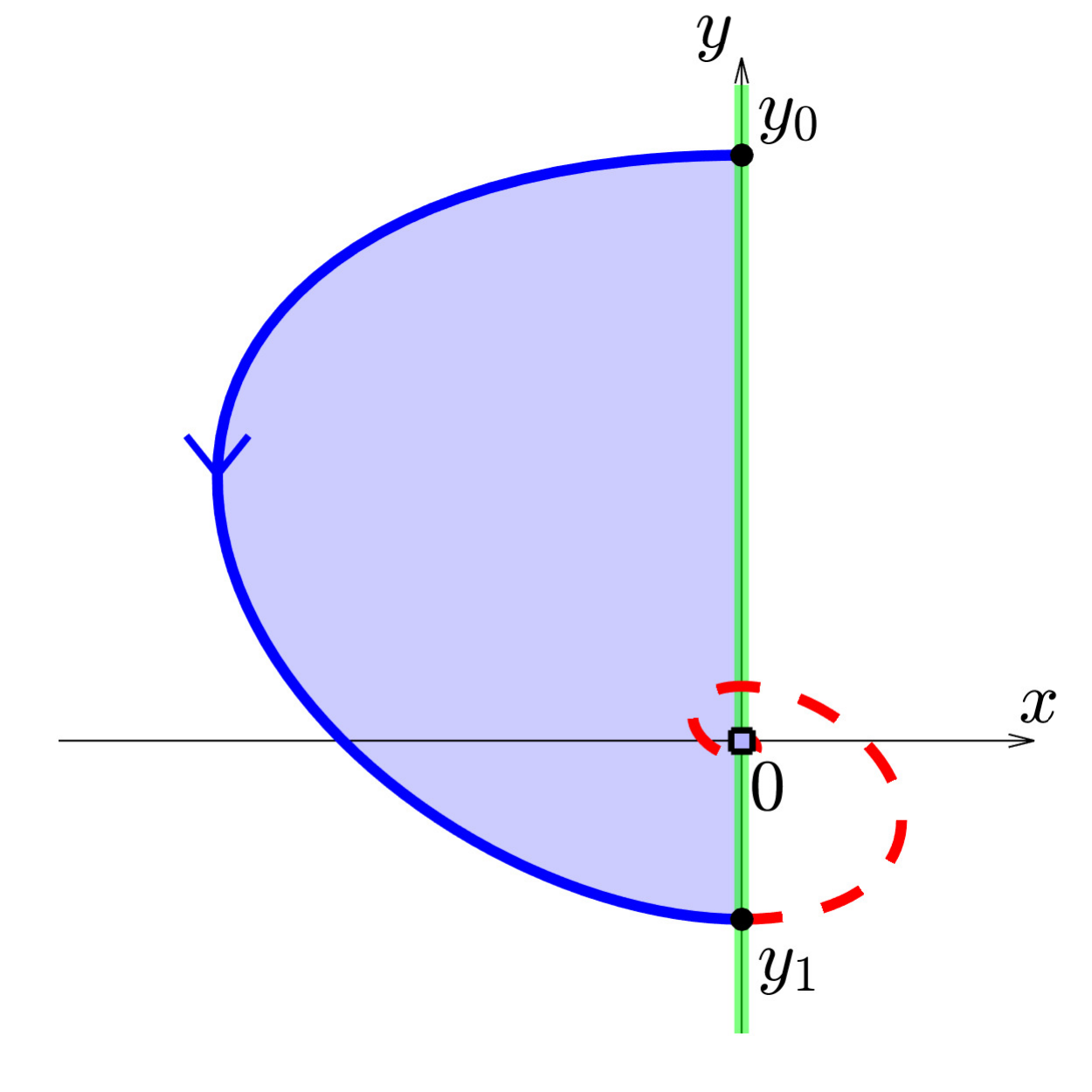}& 
     \includegraphics[width=0.33\linewidth]{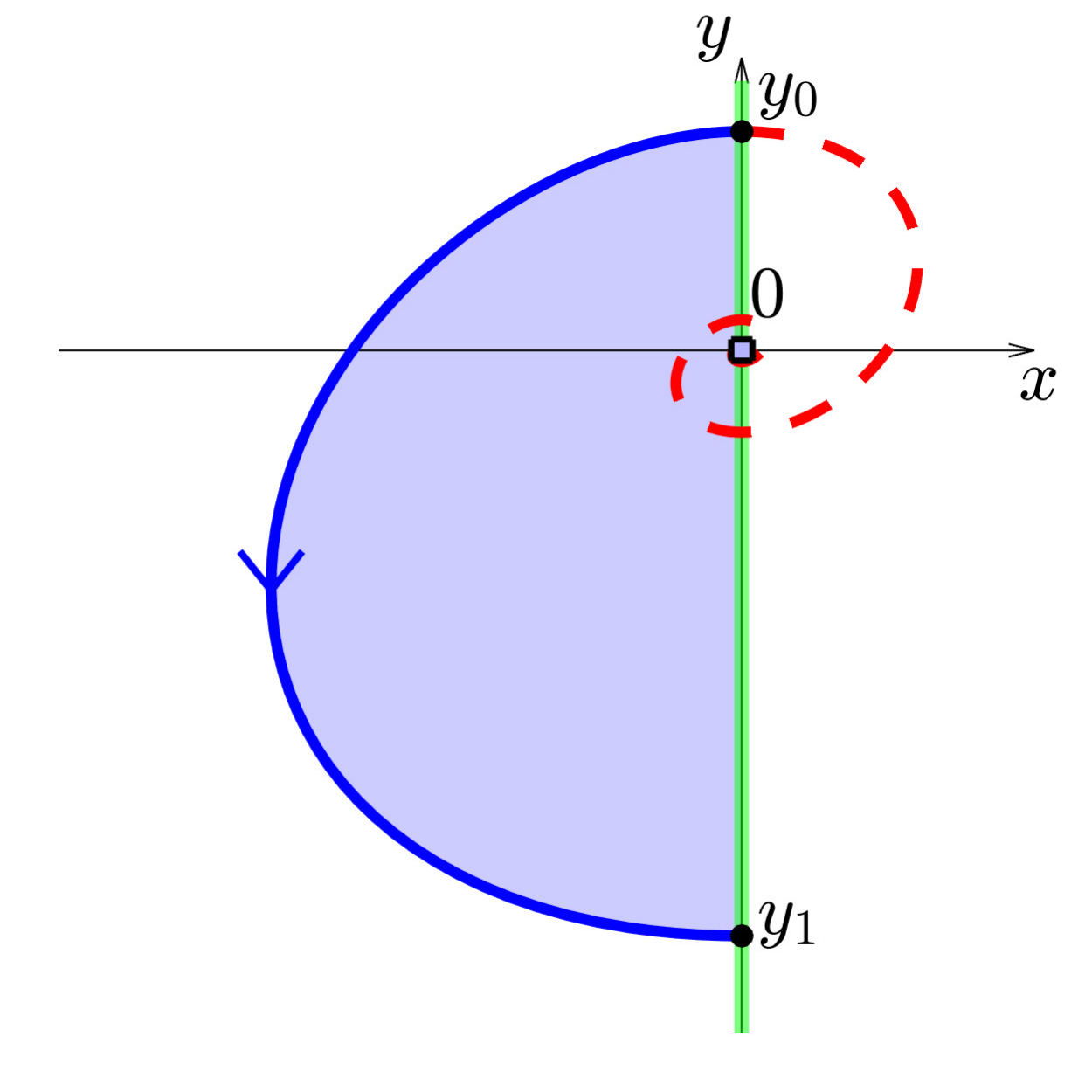}\\ 
     {\footnotesize (a)} & {\footnotesize (b)}& {\footnotesize (c)}
     \end{tabular}
    \end{center}
     \caption{Scenario (S$_1$): (a) center, (b) stable focus, (c) unstable focus.}\label{fig:casoS1}
\end{figure}

Finally, the scenario (S$_0$) includes many different cases (saddles, nodes, degenerate nodes, foci, centers, and degenerated situations without equilibria). On the one hand, for all of them there exists an invisible tangency at the origin and, thus, the left Poincar\'e half-map can be extended to $P(0)=0$ (see Remark \ref{rem:tangenciainvisible}). On the other hand, the domain and range of the left Poincar\'e half-map for foci and centers are respectively $\mathcal{D}=[0,+\infty)$ and $\mathcal{R}=(-\infty,0]$ but, for the other cases in scenario (S$_0$), the existence of invariant straight manifolds (at most two) restricts these sets. In fact, it is direct to see that invariant straight manifolds of system \eqref{eq:lienard} cannot be parallel to section $\Sigma$ and cannot contain the origin. Therefore, the intersections between all these invariant manifolds with section $\Sigma$ divide $\Sigma$ in at most three open intervals and one of them contains the origin. Thus, denoting by $\mathcal{J}$ this interval, the domain and range of the left Poincar\'e half-map for cases with invariant straight manifolds are respectively $\mathcal{D}=[0,+\infty)\cap\mathcal{J}$ and $\mathcal{R}=(-\infty,0]\cap\mathcal{J}$. See Fig.~\ref{fig:casoS0}.

\begin{figure}[!h]
    \begin{center}
    \begin{tabular}{c@{}c@{}c}
     \includegraphics[width=0.33\linewidth]{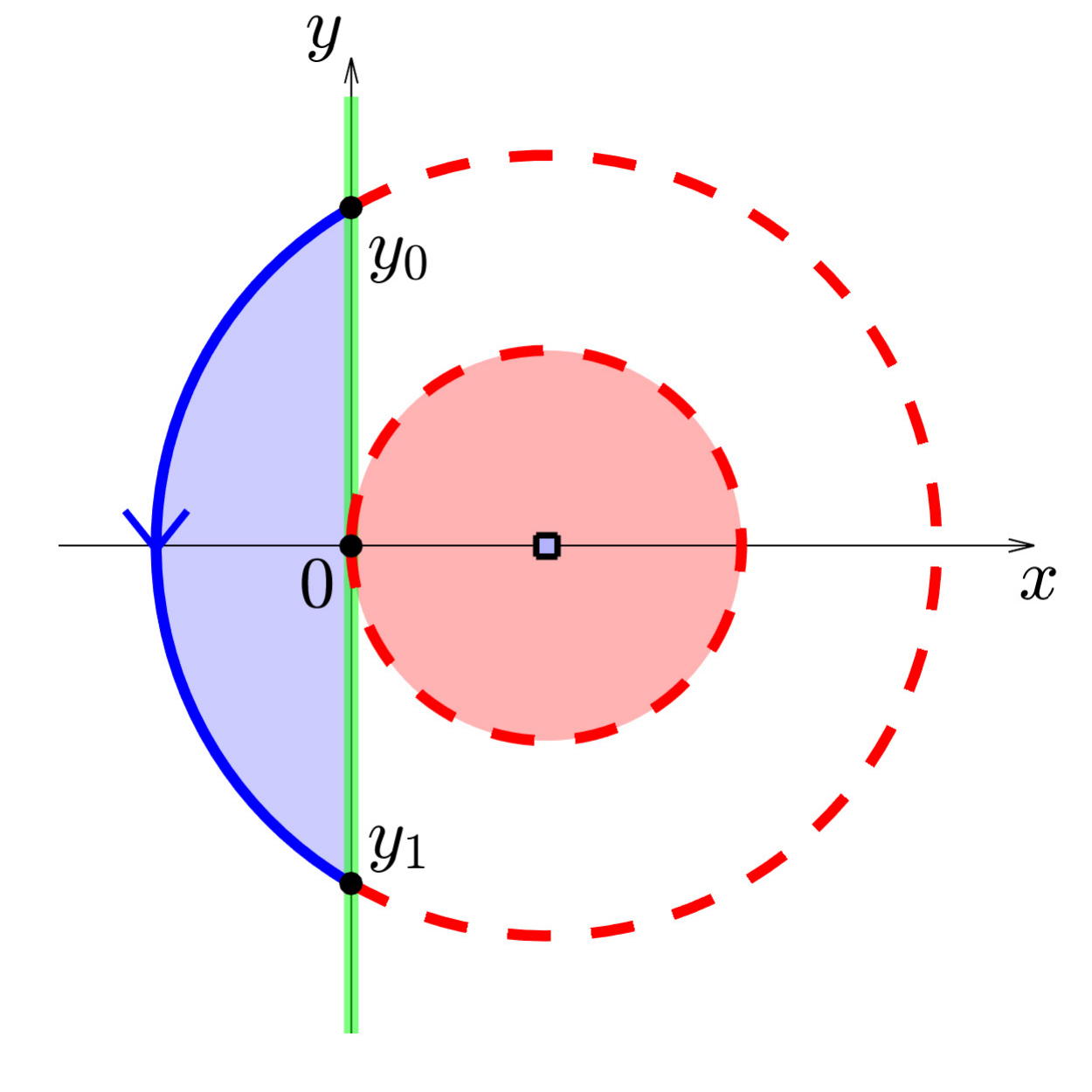}& 
     \includegraphics[width=0.33\linewidth]{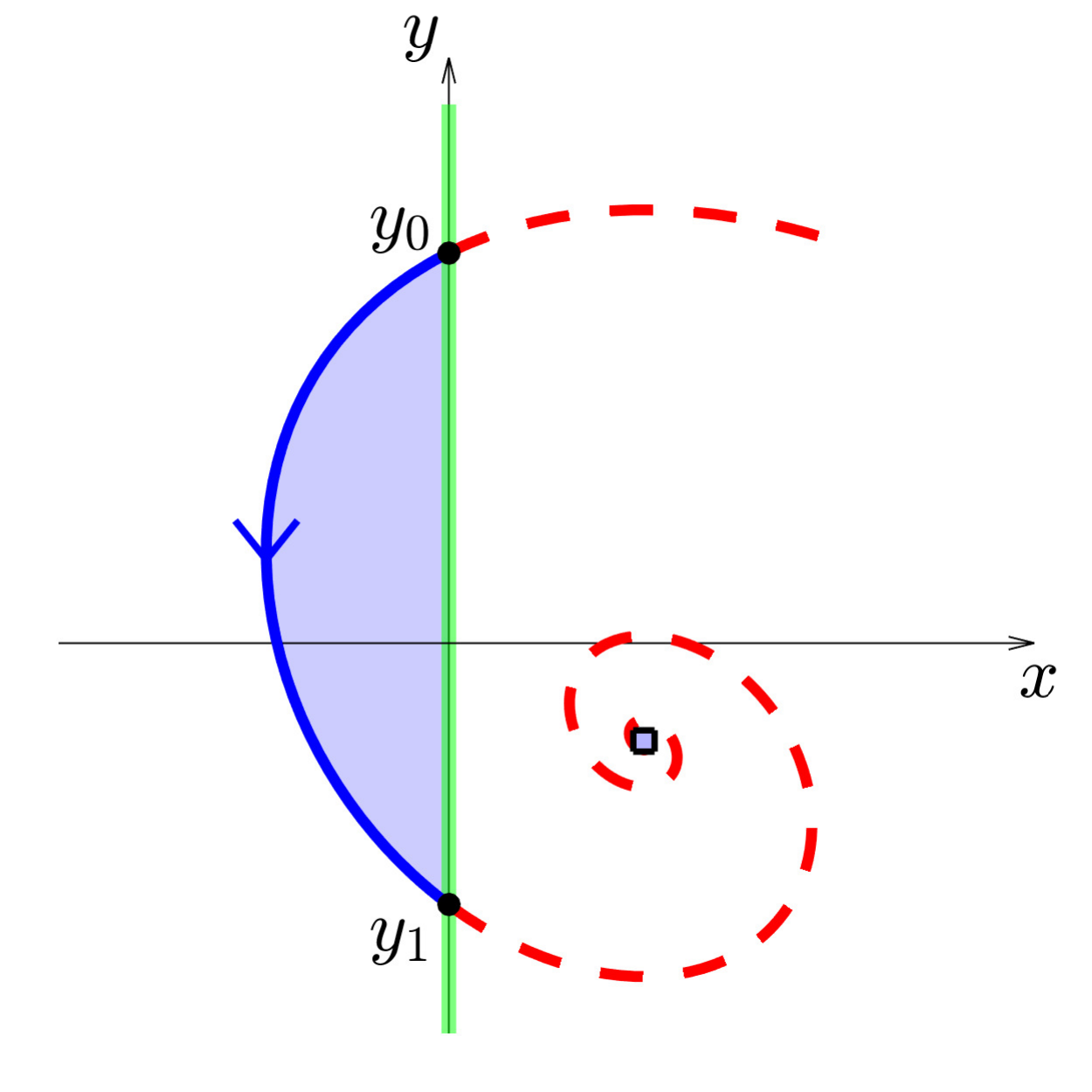}& 
     \includegraphics[width=0.33\linewidth]{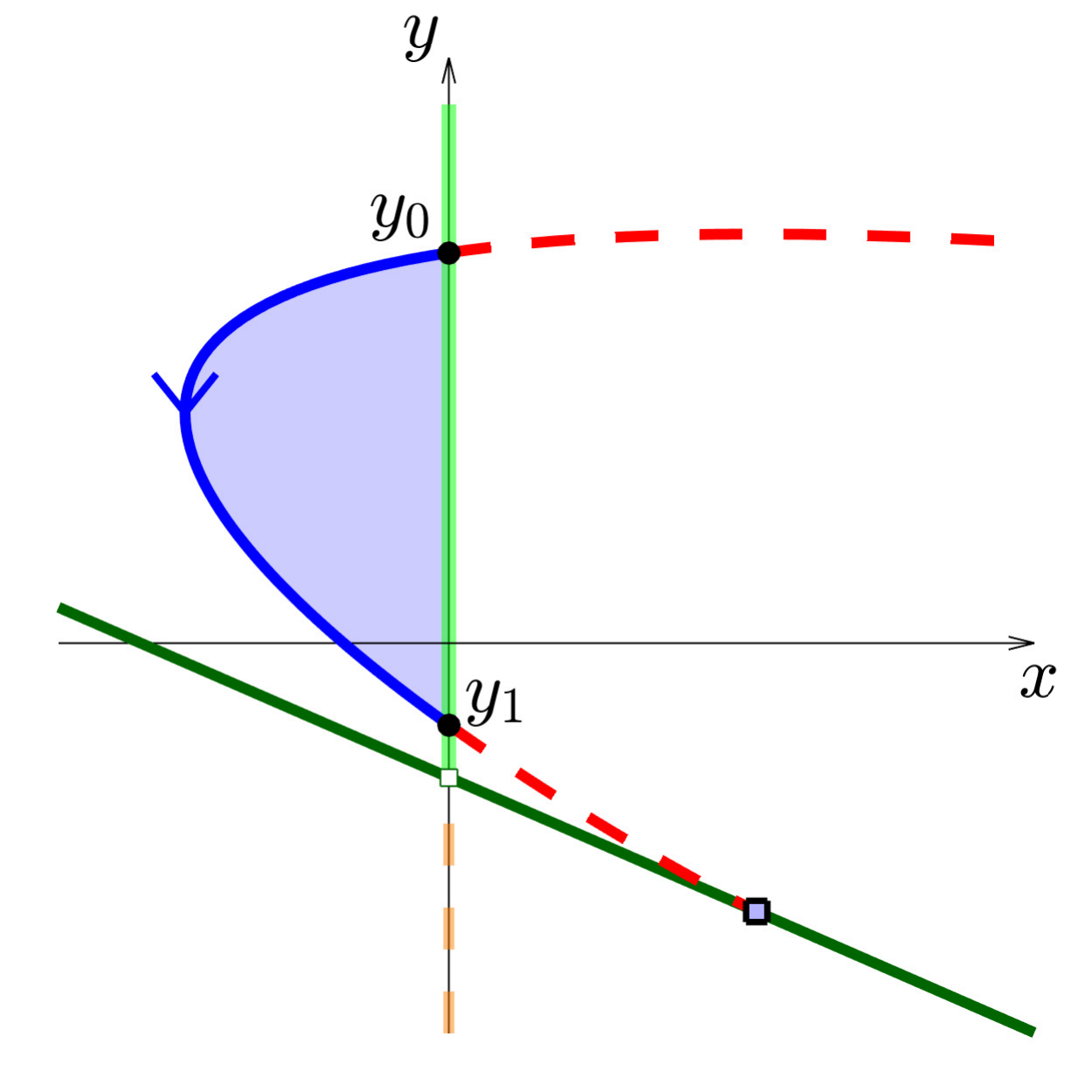}\\ 
     {\footnotesize (a)} & {\footnotesize (b)}& {\footnotesize (c)}\\
     \includegraphics[width=0.33\linewidth]{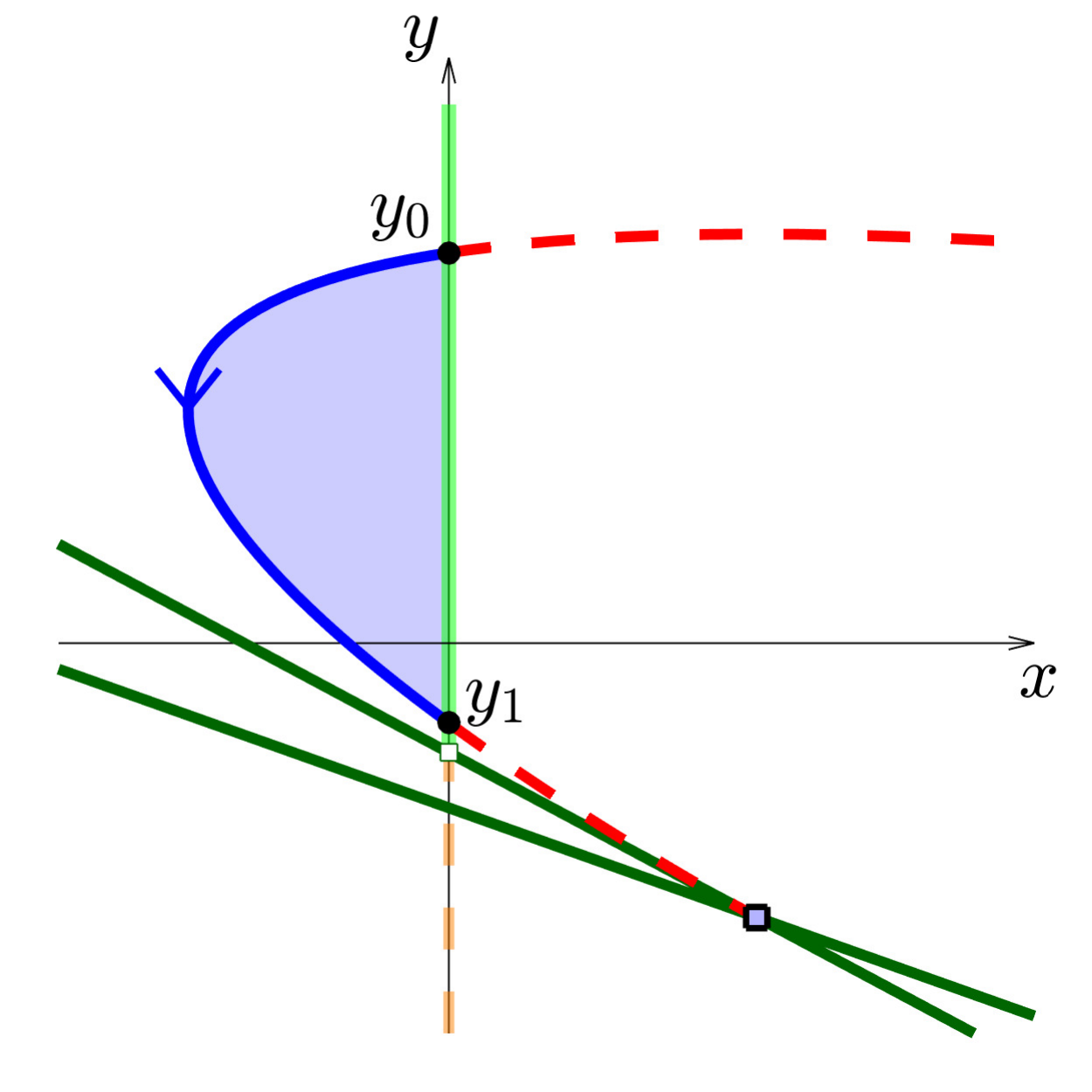}& 
     \includegraphics[width=0.33\linewidth]{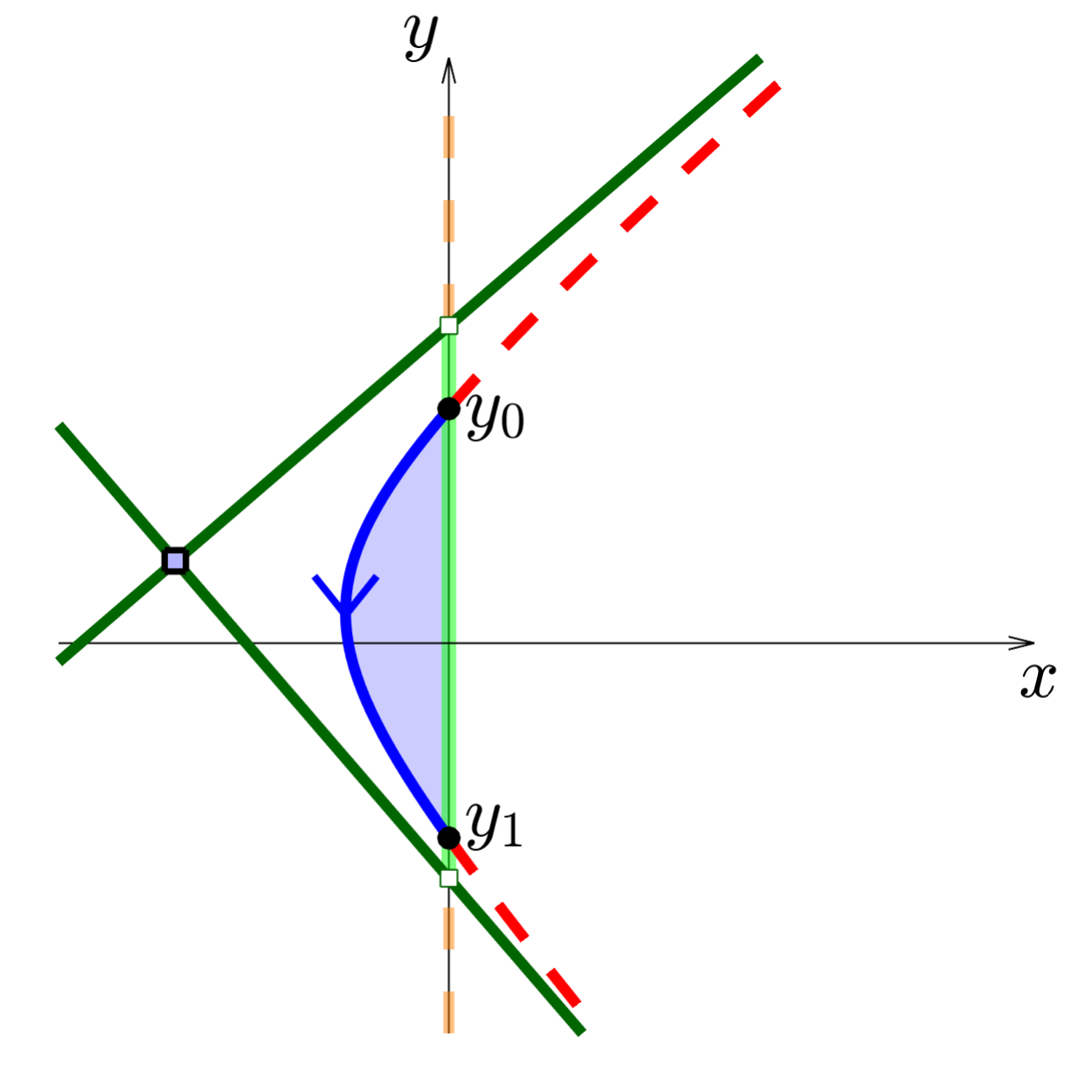}& 
     \includegraphics[width=0.33\linewidth]{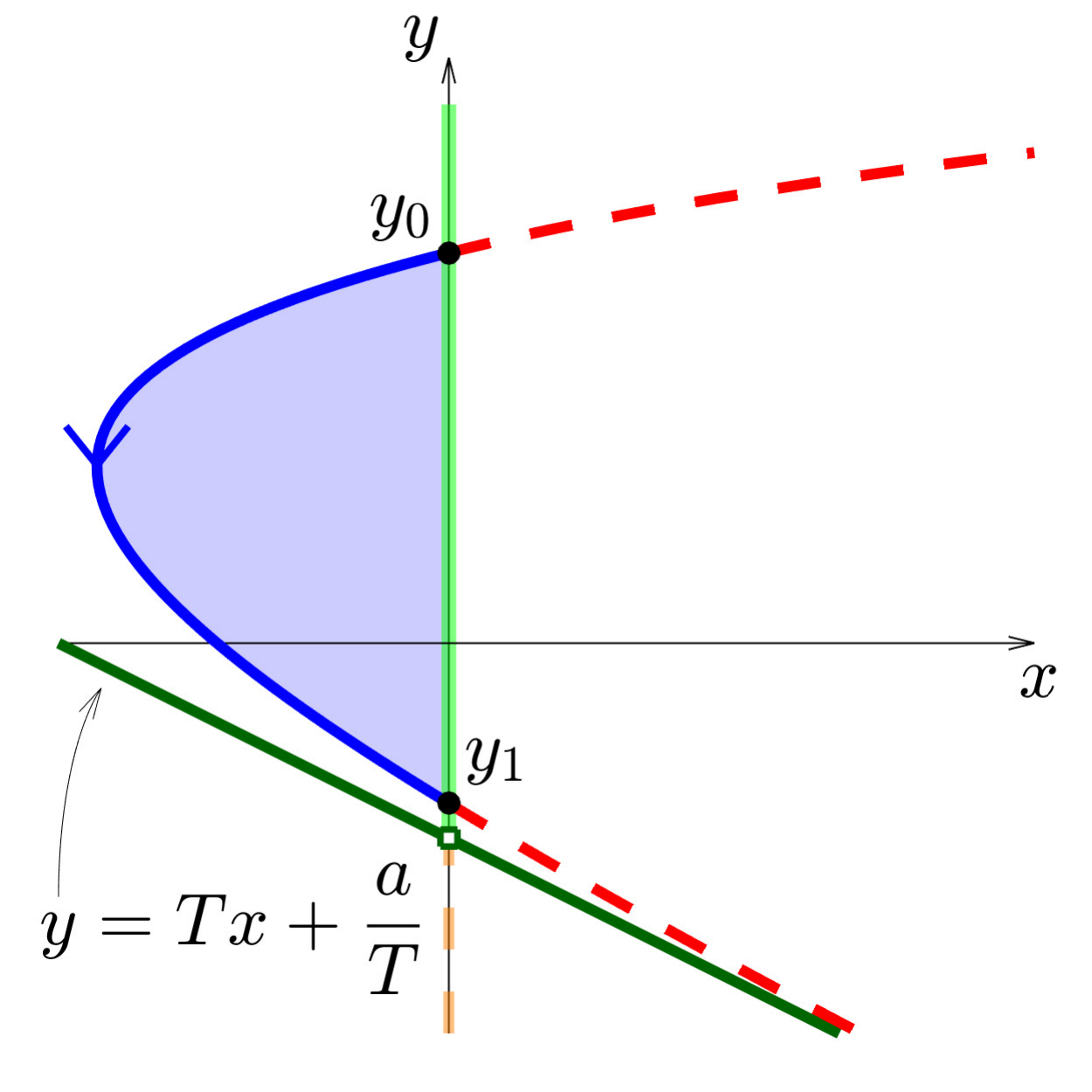}\\ 
     {\footnotesize (d)} & {\footnotesize (e)}& {\footnotesize (f)}
     \end{tabular}
    \end{center}
     \caption{Scenario (S$_0$): (a) center, (b) focus, (c) degenerate node, (d) node, (e) saddle, (f) degenerated case without equilibria.}\label{fig:casoS0}
\end{figure}

\section{Integral expression for Poincar\'e half-maps: index-like function}\label{sec:integral}

The goal of this section is to give an integral expression for the left Poincar\'e half-map of system (\ref{eq:lienard}) corresponding to section $\Sigma=\{x=0\}$. In order to achieve this aim, we use line integrals of a specific conservative vector field that is orthogonal to the flow of the system. The vector field is obtained in a convenient manner by means of a suitable inverse integrating factor.

Since inverse integrating factors are a key tool for the study of classic problems of planar smooth systems and, as far as we know, they have not been often used for piecewise systems, it is appropriate to devote a few paragraphs to present, without going into many details, some of the principal features of inverse integrating factors, particularly the basic properties and those ideas that are going to be applied to transition maps of planar linear systems. More and deeper information can be found in \cite{GarGra10}.

Let us consider the vector field $F(x,y)=(f(x,y),g(x,y))$ and the planar autonomous differential system
\begin{equation}\label{sis:fg}
\left\{\begin{array}{l}
\displaystyle\frac{dx}{dt}=f(x,y),\\
\noalign{\smallskip}
\displaystyle\frac{dy}{dt}=g(x,y),
\end{array}\right.
\end{equation}
where $f,g:\mathcal{U}\longrightarrow\mathbb{R}$ are smooth functions and $\mathcal{U}$ is a neighborhood in $\mathbb{R}^2$. 

A smooth function $V:\mathcal{U}\longrightarrow \mathbb{R}$ is an inverse integrating factor of system \eqref{sis:fg} if its zero set $V^{-1}(\{0\})=\{(x,y)\in \mathcal{U}: V(x,y)=0\}$ does not contain any non-empty open set and it satisfies the condition
\begin{equation}\label{eq:condfi}
\nabla V(x,y)\cdot F(x,y)=
V(x,y)\,\operatorname{div}F(x,y),
\end{equation}
where $\nabla V(x,y)=\left(\frac{\partial V}{\partial x}(x,y),\frac{\partial V}{\partial y}(x,y)\right)$ is the gradient of $V$,
$\mathrm{div} F(x,y)=\frac{\partial f}{\partial x}(x,y)+\frac{\partial g}{\partial y}(x,y)$ 
is the divergence of the vector field $F$ and the dot ($\cdot$) stands for the inner product.

Note that the reason why a function $V$ that satisfies the condition \eqref{eq:condfi} is called an inverse integrating factor of system \eqref{sis:fg} is that  for every $(x,y)\in \mathcal{U}\setminus V^{-1}(\{0\})$ the function $1/V$ is an integrating factor for the equation of the orbits  $g(x,y)  dx- f(x,y) dy=0$. Equivalently, the system
$$
\left\{\begin{array}{l}
\displaystyle\frac{dx}{ds}=\displaystyle\frac{f(x,y)}{V(x,y)},\\
\noalign{\smallskip}
\displaystyle\frac{dy}{ds}=\displaystyle\frac{g(x,y)}{V(x,y)},
\end{array}\right.
$$
obtained from system \eqref{sis:fg} by performing a change of the temporal variable that satisfies $ds=V(x,y)\,dt$, is 
Hamiltonian in every simply connected component of $\mathcal{U}\setminus V^{-1}(\{0\})$.

Other important results about inverse integrating factors are related to their zero sets.
Let us denote by $\Phi(t;\mathbf{p})$, the orbit of system \eqref{sis:fg} that satisfies $\Phi(0;\mathbf{p})=\mathbf{p}$. If $V$ is an inverse integrating factor of \eqref{sis:fg} then it is easy to see that the relationship
\begin{equation}\label{expr:Vphi}
V(\Phi(t;\mathbf{p}))=V(\mathbf{p})\exp\left(\int_0^t \mathrm{div} F(\Phi(s;\mathbf{p}))\,ds\right)
\end{equation}
holds.  Thus, if $V(\mathbf{p})=0$, then $V$ vanishes at the complete orbit and so the zero set of $V$ is composed of trajectories of  system \eqref{sis:fg}.

In \cite{GiaLliVia96}, it is proved that all the limit cycles of system \eqref{sis:fg} included in the domain of definition of $V$ are contained in the zero set of $V$. Moreover, under mild conditions, the separatrices of hyperbolic saddle points are also included in this set
\cite{BeGia00}. Once more, let us recommend the reading of the survey \cite{GarGra10} to deepen the knowledge of inverse integrating factors.

For the analysis in this work, a suitable inverse integrating factor must be chosen. In this case, where $L$ is the vector field given by system \eqref{eq:lienard_L}, condition \eqref{eq:condfi} is written as
\begin{equation}
\label{eq:iif_lienard}
  \nabla V \cdot L= T V,
\end{equation}
since $\operatorname{div} L=T$.

It is well-known that homogeneous linear systems have quadratic inverse integrating factors (see for instance \cite{ChaGiaGiLli99}). If the linear system is not homogeneous but has an equilibrium, a simple translation converts it into a homogeneous system so it also has a quadratic inverse integrating factor.
All the quadratic inverse integrating factors of system \eqref{eq:lienard}
are collected in the following proposition.
\begin{proposition}
The set $\mathcal{V}$ of polynomial inverse integrating factors $V(x,y)$ of degree less or equal than two for system \eqref{eq:lienard} is a finite-dimensional vector space whose dimension depends on the parameters $a$, $T$ and $D$. Concretely, the following bases $\mathcal{B}_i$ may be selected:
\begin{itemize}
\item If $a^2+D^2\neq0$ and
	\begin{itemize}
    	\item[$\circ$] $T\neq0$, then $\mathcal{B}_1=\{D^2 x^2-D T x y+D y^2+a(T^2-2 D) x-a T y+a^2\}$.
	\item[$\circ$] $T=0$, then $\mathcal{B}_2=\{1,D x^2+y^2-2 a x\}$.
	\end{itemize}
\item If $a^2+D^2=0$ and
	\begin{itemize}
    	\item[$\circ$] $T\neq0$, then $\mathcal{B}_3=\{y^2-T x y,y-Tx\}$.
	\item[$\circ$] $T=0$, then $\mathcal{B}_4=\{1,y,y^2\}$.
	\end{itemize}
\end{itemize}
\end{proposition}
\begin{proof} The proof is straightforward by imposing that the generic real quadratic polynomial in two variables
$$
V(x,y)=\sum_{0\le i+j\le2} \alpha_{ij} x^i y^j
$$
satisfies condition \eqref{eq:iif_lienard}.
\end{proof}

Except in the case $a^2+D^2\neq0$, $T\neq0$, the dimension of vector space $\mathcal{V}$ is greater than one, so there exist linearly independent inverse integrating factors. Since the division of two integrating factors is constant on orbits, the division of two linearly independent inverse integrating factors is a first integral of the system. In that case, all the orbits of the system can be obtained as the level curves of these quotients.

For $T=0$, system \eqref{eq:lienard} is reversible (invariant under the change $y\leftrightarrow-y$, $t\leftrightarrow-t$). It is also hamiltonian. In fact, constant polynomials are integrating factors (see the constant polynomial $1$ in the bases $\mathcal{B}_2$ and $\mathcal{B}_4$). From this and the previous paragraph, any inverse integrating factor is constant along the orbits of system \eqref{eq:lienard}.

In addition to these two comments, let us mention that the linear combination $a^2 \cdot 1+ D \cdot \left(D x^2+y^2-2 a x\right)$ of elements of the basis $\mathcal{B}_2$ could also be obtained from the unique element of the basis $\mathcal{B}_1$ if $T$ were allowed to vanish. As a conclusion, 
we choose
\begin{equation}
\label{eq:fii}
  V(x,y)=D^2 x^2-D T x y+D y^2+a(T^2-2 D) x-a T y+a^2
\end{equation}
as the expression of the inverse integrating factor for system \eqref{eq:lienard} under condition
\begin{equation}
  \label{eq:cond_aD}
  a^2+D^2\neq0.
\end{equation}

\begin{remark} \label{rem:elipses}
Trivially, the level curves of the inverse integrating factor $V$ are conics. In particular, when $4 D-T^2>0$, they are ellipses whose center is the equilibrium point of system \eqref{eq:lienard} and the change of variables
\begin{equation}
\label{eq:cambioelipse}
\left\{
\begin{array}{l}
\displaystyle x=X+\frac{a}{D},\\
\noalign{\smallskip}
\displaystyle y= \alpha X+ \beta Y+\frac{a T}{D}
\end{array}
\right.
\end{equation} 
for $
\alpha=T/2
$,
$
\beta=\sqrt{4 D-T^2}/2
$
transforms the inverse integrating factor into
$
\widetilde{V}(X,Y)=\beta^2 (\alpha^2+\beta^2) \left(X^2+Y^2\right)=\frac{4D^2-T^2 D}{4} \left(X^2+Y^2\right)
$. Moreover, in this case, when $T<0$ the inverse integrating factor $V$ is a Lyapunov function for system \eqref{eq:lienard}, see equation 
\eqref{eq:iif_lienard}.

\end{remark}

Note that, from now on, the study will be restricted to case \eqref{eq:cond_aD}
since for $a^2+D^2=0$ Poincar\'e half-maps to section $\Sigma=\left\{x=0\right\}$ of system \eqref{eq:lienard} cannot exist. This is an immediate conclusion from the fact that the component $y$ of every solution of system \eqref{eq:lienard} is constant and, therefore, reinjection into $\Sigma$ is not possible. Moreover, as it was said in the previous section, for $a^2+D^2=0$ the straight line $y=Tx$ is foliated by equilibrium points.

Now that a suitable inverse integrating factor has been chosen, 
as it has been said in the introduction, it is important to determine the zero set $V^{-1}(\{0\})$.
The following proposition describes it.
\begin{proposition}
\label{proposition:zeroset}
Depending on the parameters of system \eqref{eq:lienard} and under condition \eqref{eq:cond_aD}, the zero set $V^{-1}(\{0\})$ of function $V$ given in \eqref{eq:fii} is the empty set, a single point (the equilibrium point of system \eqref{eq:lienard}), a single straight line (invariant for system \eqref{eq:lienard}) or a pair of crossing straight lines (the invariant manifolds of the equilibrium point of system \eqref{eq:lienard}). Concretely:
\begin{itemize}
\item For $D=0$ (no equilibrium case) and
	\begin{itemize}
	\item[$\circ$] $T=0$, then $V^{-1}(\{0\})=\emptyset$.
	\item[$\circ$] $T\neq0$, then $V^{-1}(\{0\})=\{(x,y)\in\mathbb{R}^2: \ T^2 x-Ty+a=0\}$.
	\end{itemize}
\item For $D\neq0$ (equilibrium at $(x,y)=(a/D,aT/D)$) and
	\begin{itemize}
	\item[$\circ$] $T^2-4D>0$, then \newline $V^{-1}(\{0\})=\{(x,y)\in\mathbb{R}^2: \ 2 D \left(x-\frac{a}{D}\right)=\left(T\pm\sqrt{T^2-4D} \right) \left(y-\frac{aT}{D}\right)\}$.
	\item[$\circ$] $T^2-4D=0$, then \newline $V^{-1}(\{0\})=\{(x,y)\in\mathbb{R}^2: \ 2 D \left(x-\frac{a}{D}\right)=T \left(y-\frac{aT}{D}\right)\}$.
	\item[$\circ$] $T^2-4D<0$, then $V^{-1}(\{0\})=\{(a/D,aT/D)\}$.
	\end{itemize}
\end{itemize}
\end{proposition}

\begin{proof}
For $D=0$, condition \eqref{eq:cond_aD} implies that $a\neq0$ and the inverse integrating factor is $V(x,y)=a\left(T^2 x-Ty+a\right)$. Thus the conclusion is obvious.

For $D\neq0$, the inverse integrating factor $V(x,y)$ can be written as 
$$
  V(x,y) =-D \; \det
  \left(
  A 
  \left.
  \left(
  \begin{array}{c}
  x-\frac a D\\ \noalign{\smallskip}
  y-\frac{a T}D
  \end{array}
  \right)\right|
   \left(
  \begin{array}{c}
  x-\frac a D\\ \noalign{\smallskip}
  y-\frac{a T}D
  \end{array}
  \right)
  \right).
$$
Therefore, V vanishes if, and only if, the vector $( x-\frac a D,y-\frac{a T}D)^T$ belongs to a real eigenspace of matrix $A$. From this, the proof is direct.
\end{proof}

Since the expression given in this work for the left Poincar\'e half-map in terms of the inverse integrating factor involves divergent integrals at zero (concretely, when the equilibrium of system \eqref{eq:lienard} is located at the origin), it is necessary to use the concept of Cauchy principal value that, particularized to divergences at zero, is defined as follows. Let be $h$ a continuous function in $[\xi_1,\xi_2]\setminus\{0\}$, where $\xi_1<0<\xi_2$. The \emph{Cauchy Principal Value} (PV) of integral $\int_{\xi_1}^{\xi_2} h(\xi) d\xi$ is the following limit (if it exists):
\begin{equation}
\label{eq:PV}
  \operatorname{PV}\int_{\xi_1}^{\xi_2} h(\xi) d\xi :=
  \lim_{\varepsilon \searrow 0} \left(\int_{\xi_1}^{-\varepsilon} h(\xi) d\xi+ \int_{\varepsilon}^{\xi_2} h(\xi) d\xi\right).
\end{equation}
Obviously, if $h$ is also continuous at $\xi=0$ then the Cauchy principal value coincides with the value of the integral. By convention,
it is said that $\operatorname{PV}\int_{\xi_2}^{\xi_1} h(\xi) d\xi=-\operatorname{PV}\int_{\xi_1}^{\xi_2} h(\xi) d\xi$.
 In
\cite{GaLliMaMa00,GaMaMa02}, the Cauchy principal value is applied to the analysis of monodromy and the study of the centre problem for 
some types of planar systems.

The first main theorem of the manuscript is based on the integration of vector field 
\begin{equation}
\label{eq:campoortogonal}
G(x,y)=\left(-\frac{D x-a}{V(x,y)}, \frac{Tx-y}{V(x,y)}\right)
\end{equation}
along a suitable Jordan curve for the different scenarios given in Definition \ref{def:escenarios}.
Notice that vector field $G$ is orthogonal to the flow of system \eqref{eq:lienard} and conservative in every simply connected component
of $\mathbb{R}^2\setminus V^{-1}(\{0\})$.

\begin{remark}
\label{rem:ortogonal}
In order to choose an orientation, from now on we denote $(x,y)^{\perp}=(-y,x)$. Thus, when $V$ does not vanish,
it holds that $G=\frac{L^{\perp}}{V}$ where $L$ is the vector field defined in \eqref{eq:lienard_L}. 
Since condition \eqref{eq:iif_lienard} can be equivalently written as $L^{\perp} \cdot \nabla V^{\perp}= T V$,
the equality
\begin{equation}
\label{eq:iif_ortogonal}
G \cdot \nabla V^{\perp}= T
\end{equation}
is satisfied.
\end{remark}

\begin{remark}
\label{rem:windingnumber}
Let be $\Delta$ any piecewise-smooth planar curve that does not intersect the zero set $V^{-1}(\{0\})$ (i.e., it does not intersect the invariant manifolds or the equilibrium of system \eqref{eq:lienard}, if exist).

Let us consider the case $4D-T^2>0$ and let be $\widetilde{\Delta}$, in coordinates $(X,Y)$, the image of the curve $\Delta$ by the change of variables given in \eqref{eq:cambioelipse}. Let be $(X_a,Y_a)$ and $(X_b,Y_b)$ the first and last points of the curve $\widetilde{\Delta}$. Then, it is trivial that
$$
\begin{array}{rcl}
\int_{{\Delta}} G \cdot d\mathbf{r} \!&=&\! \displaystyle
 \frac{-1}{D}\int_{{\widetilde\Delta}}\frac{X \,dX+YdY}{X^2+Y^2}+
\frac{T}{D\sqrt{4D-T^2}} \left( 
\int_{\widetilde{\Delta}}  \frac{X \,dY-Y \, dX}{X^2+Y^2} \right)\\
\noalign{\smallskip}
 \!&=&\! \displaystyle
 \frac{-1}{D} \log\left.\left(X^2+Y^2\right)\right|_{(X_a,Y_a)}^{(X_b,Y_b)}+
 \frac{T}{D\sqrt{4D-T^2}} \left( 
\int_{\widetilde{\Delta}}  \frac{X \,dY-Y \, dX}{X^2+Y^2} \right).
 \end{array}
$$

Notice that when $V|_{\Delta}$ is constant or $\Delta$ is a closed curve, the first summand vanishes and
$$
\int_{{\Delta}} G \cdot d\mathbf{r} = \displaystyle
 \frac{T}{D\sqrt{4D-T^2}} \left( 
\int_{\widetilde{\Delta}}  \frac{X \,dY-Y \, dX}{X^2+Y^2} \right).
$$
Therefore, when $\Delta$ is a closed curve then 
$$
\oint_{{\Delta}} G \cdot d\mathbf{r} = \displaystyle
 \frac{2 \pi T}{D\sqrt{4D-T^2}} \left( \frac{1}{2 \pi}
\oint_{\widetilde{\Delta}}  \frac{X \,dY-Y \, dX}{X^2+Y^2} \right),
$$
that is, $2 \pi T\left(D\sqrt{4D-T^2}\right)^{-1}$ times the index (or winding number) of the curve $\widetilde{\Delta}$ around $(X,Y)=(0,0)$
or, equivalently, times the index of the curve $\Delta$ around the equilibrium point.

For the case $4D-T^2\leqslant0$, when $\Delta$ is a closed curve, due to the conservativeness of vector field $G$, it is clear that 
$$
\oint_{{\Delta}} G \cdot d\mathbf{r} =0.
$$
This can also be understood as the index of the curve $\Delta$ around the equilibrium, if it exists, or any other point not surrounded by $\Delta$.
\end{remark}

Now, we are in a position to present and prove a theorem that states a new way to write, in terms of an integral expression, the existing relationship between a value and its image by means of the left Poincar\'e half-map.

\begin{theorem}
\label{th:implicaderecha}
Let assume that condition \eqref{eq:cond_aD} and hypothesis (H) hold. Let be $y_1=P(y_0)$ the image of $y_0$ by the left Poincar\'e half-map, 
$V$ the inverse integrating factor given in expression \eqref{eq:fii}, 
$\Gamma$ the Jordan curve given in Eq.~\eqref{eq:Jordancurve} and (S$_k$), $k\in\{0,1,2\}$,
the corresponding scenario given in Definition \ref{def:escenarios}. Then
\begin{equation}
\label{eq:curvasnivel}
  \operatorname{PV}\int_{y_1}^{y_0} \frac{-y}{V(0,y)} dy =d_k, 
\end{equation}
where
$$
  \displaystyle 
  d_k=\left\{
  \begin{array}{cl}
  0 & \mbox{if } \ k=0,\\
  \noalign{\medskip}
  \displaystyle  \frac{k\pi T}{D\sqrt{4D-T^2}} & \mbox{if } \ k=1,2.
  \end{array}
  \right.
$$
\end{theorem}
\begin{proof}
The proof of this theorem is direct from the computation of the line integral of vector field
\eqref{eq:campoortogonal} along the Jordan curve $\Gamma$ given in Eq.~\eqref{eq:Jordancurve} and positively oriented.
This computation depends on the relative position of $\Gamma$ and the equilibrium of system
\eqref{eq:lienard}, if any. Therefore, the proof is divided into three parts.

Before that, since the vector field $G$ is orthogonal to the flow of system \eqref{eq:lienard}
at $\mathbb{R}^2\setminus V^{-1}(\{0\})$ then 
\begin{equation}
\label{eq:gamma2=0}
\int_{\Gamma_2} G \cdot d\mathbf{r}=0
\end{equation}
due to $\Gamma_2$ is a piece of an orbit of the system. 

Let us begin with the easiest case, that is, scenario (S$_0$) (see Fig.~\ref{fig:implicaderecha}(a)). 
From Proposition \ref{proposition:zeroset}, it is obvious that $V$ does not vanish in $\operatorname{Int}(\Gamma)\cap\Gamma$.
Thus, since the vector field $G$ is conservative, the integral
$$
   \oint_{\Gamma} G \cdot d\mathbf{r}=\int_{\Gamma_1} G \cdot d\mathbf{r}+\int_{\Gamma_2} G \cdot d\mathbf{r}
$$
vanishes. Besides that, identity \eqref{eq:gamma2=0} implies that 
$$
0=\oint_{\Gamma} G \cdot d\mathbf{r}=\int_{\Gamma_1} G \cdot d\mathbf{r}=\int_{y_1}^{y_0} \frac{-y}{V(0,y)} dy.
$$
The proof for scenario (S$_0$) is finished.

Regarding scenario (S$_2$) (see Fig.~\ref{fig:implicaderecha}(b)), the unique equilibrium point belongs to the interior of the Jordan curve $\Gamma$ and it is the only point at which the inverse integrating factor $V$ vanishes. Therefore, by using identity \eqref{eq:gamma2=0} and 
Remark \ref{rem:windingnumber} 
it is trivial to see that
$$
\int_{y_1}^{y_0} \frac{-y}{V(0,y)} dy=\int_{\Gamma_1} G \cdot d\mathbf{r}=\oint_{\Gamma} G \cdot d\mathbf{r}=
\frac{2 \pi T}{D\sqrt{4D-T^2}}.
$$

The last part of the proof, corresponding to scenario (S$_1$), where $a=0$ and $D\neq0$, is a little bit more complicated. In fact, it is the reason that motivates us to use the Cauchy principal value as defined in Eq.~\eqref{eq:PV} because the improper integral
$$
\int_{y_1}^{y_0}  \frac{-y}{V(0,y)} dy=\int_{y_1}^{y_0} \frac{-1}{D y}  dy
$$
is divergent.

Since the unique equilibrium of the system (the origin) is located at the Jordan curve $\Gamma$, 
it is not possible to proceed as in previous scenarios. Under these circumstances,  it is usual to choose a Jordan curve
$
\widetilde{\Gamma}=\Gamma_{11}\cup\Gamma_2\cup\Gamma_{12}\cup\Gamma_3
$
as shown in Fig.~\ref{fig:implicaderecha}(c). Note that $\Gamma_3$ intersects the Poincar\'e section $\{x=0\}$ at the
symmetric points $(0,-\varepsilon)$ and $(0,\varepsilon)$, with $\varepsilon>0$. Once more, by identity \eqref{eq:gamma2=0} and 
Remark \ref{rem:windingnumber}, 
it follows that
$$
   \int_{y_1}^{-\varepsilon}  \frac{-1}{D y} dy+\int_{\varepsilon}^{y_0}  \frac{-1}{D y} dy=
   -\int_{\Gamma_3} G \cdot d\mathbf{r}=\frac{\pi T}{D\sqrt{4D-T^2}}.
$$
Now, to finish the proof, it is enough to take limits as $\varepsilon\searrow0$. 
\end{proof}

\begin{figure}[!h]
    \begin{center}
    \begin{tabular}{c@{}c@{}c}
     \includegraphics[width=0.33\linewidth]{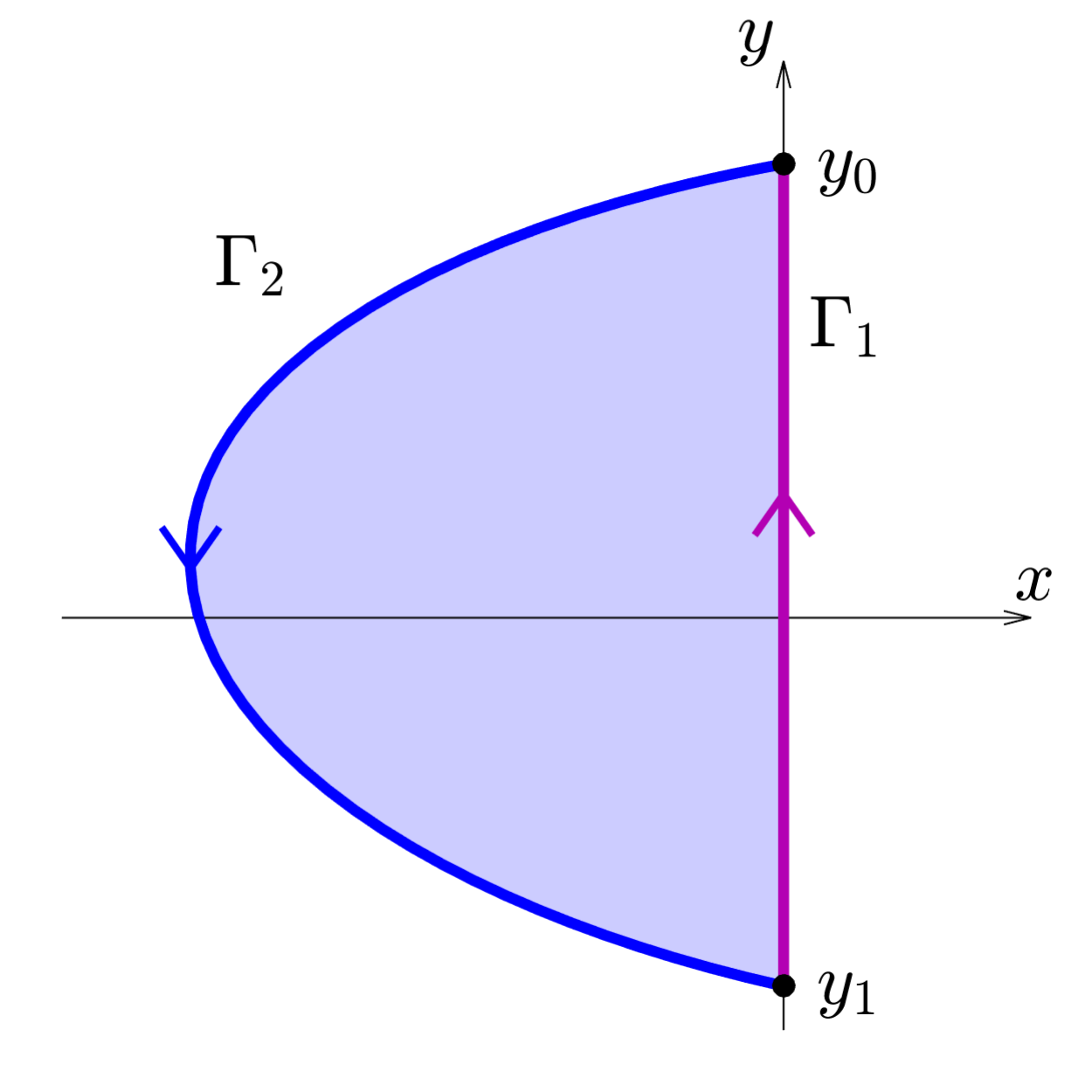}&
     \includegraphics[width=0.33\linewidth]{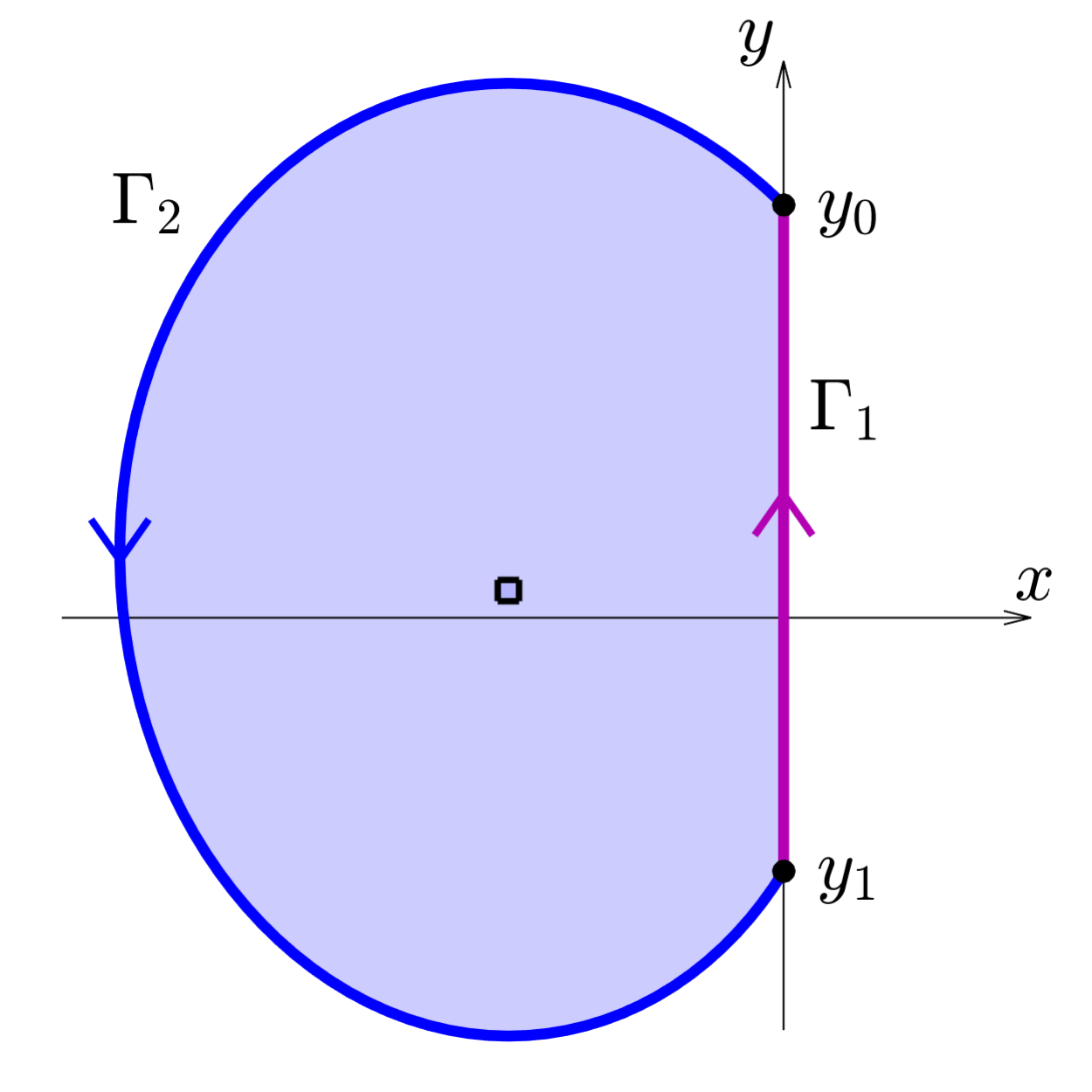}&
     \includegraphics[width=0.33\linewidth]{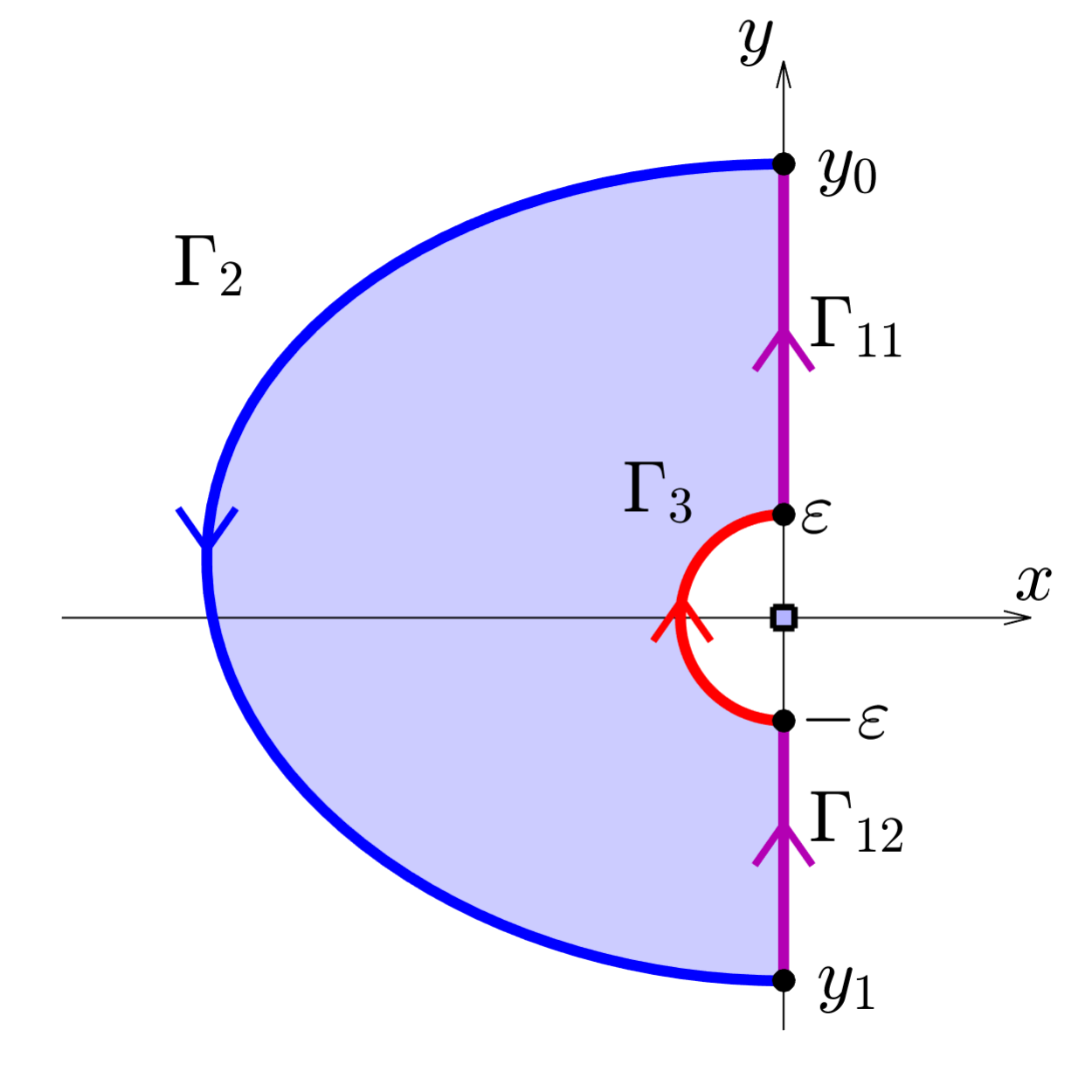}\\
     {\footnotesize (a)} & {\footnotesize (b)}& {\footnotesize (c)}
    \end{tabular}
    \end{center}
     \caption{Curves used in the proof of Theorem \ref{th:implicaderecha}. Labels (a), (b), (c) stand respectively for scenarios (S$_0$), (S$_2$) and (S$_1$). }\label{fig:implicaderecha}
\end{figure}

\begin{remark} 
{ Analogously, if $y_0$ is the image of $y_1$ by the right Poincar\'e half-map, the same expression given in equation 
\eqref{eq:curvasnivel} is obtained but now $d_k$ must be changed by $-d_k$.}
\end{remark}

{ Note that, at this point, 
the flight time has been removed from the expression for the Poincar\'e half-maps and the analysis of the many cases that appear due to the configuration of the spectra of matrix $A$ have been reduced to the study of the level curves of the single function,
\begin{equation}\label{eq:F}
 \mathcal{F}(y_1,y_0)=\operatorname{PV}\int_{y_1}^{y_0} \frac{-y}{V(0,y)} dy,
\end{equation}
where $V$  is the inverse integrating factor given in expression \eqref{eq:fii}.
}

{
Let be $\Delta_2$ a piece of an orbit of system \eqref{eq:lienard} that starts at $(0,y_0)$
and ends at $(0,y_1)$,
$\Delta_1$ the segment bounded by $y_0$ and $y_1$ in section $\Sigma$, and  $\Delta=\Delta_1\cup\Delta_2$ the corresponding positively oriented closed curve. Reasoning as in the proof of Theorem \eqref{th:implicaderecha}, it is trivial that
$$
   \mathcal{F}(y_1,y_0)=\oint_{{\Delta}} G \cdot d\mathbf{r} 
$$
holds. A suitable name for function $\mathcal{F}$ comes from Remark \ref{rem:windingnumber}.
}

\begin{definition}
\label{def:index-like function}
 Function $\mathcal{F}$ given in equation \eqref{eq:F} is called the index-like function.
\end{definition}

\begin{remark}
A logical consequence of the direct procedure (that is, not depending on the spectra of matrix $A$) of construction of an alternative way of writing the left Poincar\'e half-map is that the index-like function is common for all cases and it only depends (through the expression of $V(0,y)$) on the parameters of the linear system \eqref{eq:lienard} and not on the eigenvalues of matrix $A$. Since the qualitative information of such a linear system is given by its characteristic polynomial $\mathcal{P}_A$ and the location of the equilibrium, if it exists, it seems to be clear that there must exist a relationship between $V(0,y)$ and the characteristic polynomial. In fact, it is trivial to prove the equality 
$$
y^2 \mathcal{P}_A\left(\frac{a}{y}\right)=V(0,y)
$$
for $y\neq0$. Thus, the index-like function given in Definition \ref{def:index-like function} (and, consequently, the Poincar\'e half-maps) can be written in terms of the characteristic polynomial  
$\mathcal{P}_A$.
\end{remark}

The next step of this work will consist on the study of the index-like function $\mathcal{F}$, with a focus on the properties that can be extrapolated to Poincar\'e half-maps.

\section{Analysis of the index-like function}\label{sec:index-like}

The first important task in the analysis of the index-like function $\mathcal{F}$ given in \eqref{eq:F} is to delimit its domain of definition to avoid nontrivial zeros of the function $V(0,y)$. Although the integral in the definition of the function $\mathcal{F}$ could have been extended, via the Cauchy principal value, to be valid for intervals containing those zeros, in the context of this work is not necessary to do it because the Poincar\'e half-maps of system \eqref{eq:lienard} cannot be defined in them.

In order to delimit the domain of function $\mathcal{F}$, 
let us consider the open interval 
\begin{equation}
\label{eq:I}
\mathcal{I}=\Big\{y\in\mathbb{R}: \ (0,y)\in \operatorname{int}\left(\mathcal{C}\cup\left\{(0,0)\right\}\right)\Big\},
\end{equation}
where $ \operatorname{int}(\cdot)$ stands for the topological concept of the interior of a set and the set
$\mathcal{C}$ is the unique maximal connected component of $\mathbb{R}^2\setminus V^{-1}(\{0\})$ that contains
two points $(0,\xi_1)$ and $(0,\xi_2)$ with $\xi_1 \cdot \xi_2 <0$. 
On the one hand, note that for $a\neq0$ interval $\mathcal{I}$ could have been equivalently defined as $\mathcal{I}=\{y\in\mathbb{R}: \ (0,y)\in \mathcal{C}\}$ because $(0,0)\notin V^{-1}(\{0\})$.
On the other hand, when $a=0$, that is, when the origin $(0,0)\in V^{-1}(\{0\})$ or equivalently the only equilibrium point is the origin (remember that, from \eqref{eq:cond_aD}, it is $a^2+D^2\neq0$), the set $\mathcal{C}$ and so the interval $\mathcal{I}$ are $\mathcal{C}=\mathbb{R}^2\setminus\{(0,0)\}$ and $\mathcal{I}=\mathbb{R}$ when the equilibrium is a focus or a center but they are empty when the equilibrium it is a saddle or a node. At this moment, it should have to be obvious that the intricate definition of the interval $\mathcal{I}$ is suitable with the intention of removing the cases where the existence of a Poincar\'e half-map is not allowed.  

Moreover, it is clear that the inverse integrating factor 
$V$ is strictly positive on the set $\mathcal{C}$ when $\mathcal{C}\neq\emptyset$. Therefore, the inequality $V(0,y)>0$ holds for all $y\in\mathcal{I}$ when $a\neq0$ and for
all $y\in\mathcal{I}\setminus\{0\}$ when $a=0$.

Finally, when $\mathcal{I}$ is not empty, it can be written as $\mathcal{I}=(\mu_1,\mu_2)$ where $-\infty\leqslant\mu_1<0<\mu_2\leqslant+\infty$. Besides that, if $-\infty<\mu_1$ (resp. $\mu_2<+\infty$) then $V(0,\mu_1)=0$ (resp. $V(0,\mu_2)=0$).

Let us consider now the integrating function from the definition of the function $\mathcal{F}$ given in \eqref{eq:F},
\begin{equation}
\label{eq:integrando}
h(y)=\frac{-y}{V(0,y)}=\frac{-y}{D y^2-a T y+a^2} \ ,
\end{equation}
defined for every $y\in\mathcal{I}$ when $a\neq0$ and for every
$y\in\mathcal{I}\setminus{\{0\}}$ when $a=0$.  This function $h$ is strictly positive for $y<0$ and strictly negative for $y>0$.

Since for $i=1,2$ the endpoints $\mu_i$ of interval $\mathcal{I}$ are infinity or satisfy $V(0,\mu_i)=0$, any integral of $h$ involving any of these points is an improper integral and it is simple to see that it is, moreover, divergent. In fact, for every $a\neq0$ and $z\in\mathcal{I}$, the following equalities hold
\begin{equation}
\label{eq:divergenciaextremos}
\int_z^{\mu_2} h(y) dy=-\infty \quad \mbox{and}\quad \int_{\mu_1}^z h(y) dy=+\infty.
\end{equation}
Also, for $a=0$, it is $h(y)=\frac{-1}{Dy}$ and since $D$ cannot vanish any integral involving $y=0$ is also a divergent improper integral.

In this first step, the sign of $V$ has allowed to obtain an open interval $\mathcal{I}$ outside of which 
a Poincar\'e half-map could not be defined. Therefore, from now on, the study of function $\mathcal{F}$ will be restricted to 
the open square $(y_1,y_0)\in\mathcal{I}^2=\mathcal{I}\times\mathcal{I}$. The next result describes the region of analyticity of function  $\mathcal{F}$ in this square.

\begin{lemma}
\label{lemma:analitica}
Let us assume that interval $\mathcal{I}$ given in \eqref{eq:I} is not empty. Function $\mathcal{F}$ given in \eqref{eq:F} is analytic in:
\begin{enumerate}
\item \label{lemma:analitica1} the open square $\mathcal{I}^2$ for $a\neq0$;
\item 
the open set $\mathcal{I}^2\setminus\left\{(y_1,y_0)\in\mathbb{R}^2 : y_1 \cdot y_0=0\right\}$ for $a=0$.
\end{enumerate}
\end{lemma}
\begin{proof}
For $a\neq0$, function $h(y)$ given in \eqref{eq:integrando} is analytic in $\mathcal{I}$. Therefore, statement \ref{lemma:analitica1} is direct.

Trivially, for $a=0$, 
\begin{equation}
\label{eq:fparaa=0}
\mathcal{F}(y_1,y_0)=\operatorname{PV}\int_{y_1}^{y_0} \frac{-1}{D y} \,dy=\frac1D \log\left|\frac{y_1}{y_0}\right|
\end{equation}
and the proof is finished.
\end{proof}

\begin{remark} 
\label{remark:multivaluada}
It is obvious that for $a=0$, function $\mathcal{F}$ is not defined at set $\left\{(y_1,y_0)\in\mathbb{R}^2 : y_1 \cdot y_0=0\right\}$. 
Nevertheless, although
$$
  \mathop{\lim_{y_1\to 0}}_{y_0\neq0} \mathcal{F}(y_1,y_0)=-\operatorname{sign}(D) \cdot \infty \quad \mbox{and} \quad
  \mathop{\lim_{y_0\to 0}}_{y_1\neq0} \mathcal{F}(y_1,y_0)=+\operatorname{sign}(D) \cdot \infty,
$$
the directional limit along the straight line $y_1=c y_0$, for $c\neq0$, is
$$
  \lim_{y_0\to 0} \mathcal{F}(c y_0,y_0)=\frac{\log|c|}{D}
$$
so, roughly speaking, we could say that the expression $\mathcal{F}(0,0)$ takes any value at this point.
\end{remark}

Remember that the Poincar\'e half-maps correspond to specific level curves of function $\mathcal{F}$. Hence, the next result is devoted to describing the whole set of level curves $\mathcal{F}(y_1,y_0)=q\in\mathbb{R}$ restricted to $\mathcal{I}^2$. In fact, we will see that these curves can be seen as graphs of real analytic functions defined in $\mathcal{I}$  that, moreover, are solutions of the same differential equation. 

\begin{theorem}
\label{th:curvasnivel}
Let us consider the index-like function $\mathcal{F}$ given in equation \eqref{eq:F} and the inverse integrating factor $V$
given in \eqref{eq:fii}.
Let us assume that condition \eqref{eq:cond_aD} holds and that the open interval $\mathcal{I}=(\mu_1,\mu_2)$ defined in \eqref{eq:I} is not empty.  For every value $q\in\mathbb{R}$ there exist two different real analytic functions $\phi_q,\varphi_q:\mathcal{I}\longrightarrow\mathcal{I}$ such that the level curve $\mathcal{C}_q=\left\{(y_1,y_0)\in\mathcal{I}^2: \ \mathcal{F}(y_1,y_0)=q\right\}$ is the union of the graphs of functions $\phi_q$ and $\varphi_q$. Moreover, both functions are solutions of the differential equation
\begin{equation}
\label{eq:ode}
y_1 V(0,y_0) \, dy_1 -y_0 V(0,y_1)\, dy_0=0
\end{equation}
in $\mathcal{I}$.
To be more precise:
\begin{enumerate}
\item \label{th:it:a=0} For $a=0$ then $\mathcal{C}_q=\left\{(y_1,y_0)\in\mathcal{I}^2: \ (y_1-\phi_q(y_0)) (y_1-\varphi_q(y_0))=0 \right\}$, where
$\phi_q(z)=e^{Dq} \, z$ and $\varphi_q(z)=-e^{Dq} \, z$ for every $z\in\mathcal{I}$.
\item \label{th:it:aneq0.q=0} For $a\neq0$ and $q=0$ then 
$$\mathcal{C}_0=\left\{(y_1,y_0)\in\mathcal{I}^2: \ (y_1-\phi_0(y_0)) (y_1-\varphi_0(y_0))=0 \right\}$$
	where $\phi_0=\operatorname{id}$, $\varphi_0$ is an involution in $\mathcal{I}$ and $\varphi_0(0)=0$. 
	
\item \label{th:it:aneq0.qneq0} For $a\neq0$ and $q\neq 0$ then 
$$\mathcal{C}_q=\left\{(y_1,y_0)\in\mathcal{I}^2: \ (y_1-\phi_q(y_0)) (y_1-\varphi_q(y_0))=0 \right\} \ \mbox{for} \ q>0$$ and 
	$$\mathcal{C}_q=\left\{(y_1,y_0)\in\mathcal{I}^2: \ (y_0-\phi_q(y_1)) (y_0-\varphi_q(y_1))=0 \right\} \ \mbox{for} \ q<0,$$ where
	 every function $\varphi_q$ is unimodal in $\mathcal{I}$, function $\operatorname{sgn}(q) \varphi_q$ 
	has a strictly negative maximum at the origin and the equivalence $\phi_q\equiv\varphi_{-q}$ holds.
	Furthermore, the restricted functions	
\begin{equation}
\label{eq:restrictedfunctions}
	\begin{array}{l}	
        \varphi_q:[0,\mu_2) \longrightarrow (\mu_1,\varphi_q(0)], \text{ for } q>0,\\
        \varphi_q:(\mu_1,0] \longrightarrow [\varphi_q(0),\mu_2), \text{ for } q<0,
	\end{array}
\end{equation}
	are bijective.
\end{enumerate}
\end{theorem}
\begin{proof}
For $a=0$ the proof of statement \ref{th:it:a=0} is direct from expression \eqref{eq:fparaa=0} and Remark \ref{remark:multivaluada}. Moreover, since $D\neq0$, differential equation \eqref{eq:ode} is just $y_1  y_0 ( y_0 \,dy_1-y_1 \, dy_0)=0$, which is trivially satisfied by straight lines that pass through the origin.

In the rest of the proof, we assume $a\neq0$
and, as it has been said before, the principal value can be removed from the definition of the function $\mathcal{F}$. Therefore, the partial derivatives
of the function $\mathcal{F}$ are
\begin{equation}
\label{eq:derivadasF}
   \frac{\partial \mathcal{F}}{\partial y_1}(y_1,y_0)=\frac{y_1}{V(0,y_1)} \quad \mbox{and} \quad  \frac{\partial \mathcal{F}}{\partial y_0}(y_1,y_0)=\frac{-y_0}{V(0,y_0)}
\end{equation}
and so, for every $q\in\mathbb{R}$, any function implicitly defined in $\mathcal{I}$ by the equation $\mathcal{F}(y_0,y_1)=q$, if it exists, satisfies
$$
  0=d \mathcal{F}=\frac{y_1}{V(0,y_1)} \, dy_1-\frac{y_0}{V(0,y_0)} \, dy_0,
$$
what is equivalent to differential equation \eqref{eq:ode} since $V(0,z)>0$ for every $z\in\mathcal{I}=(\mu_1,\mu_2)$.

Let us prove now the existence of functions $\phi_q$ and $\varphi_q$ mentioned in the statement of the theorem. From the values of the partial derivatives of $\mathcal{F}$ it is clear that for every $y_0\in\mathcal{I}$, the function
$\mathcal{F}(\,\cdot\,,y_0)$ is strictly decreasing at interval $(\mu_1,0)$ and strictly increasing at interval $(0,\mu_2)$. From the equalities given in \eqref{eq:divergenciaextremos} it is clear that for every $y_0\in\mathcal{I}$, 
$$
\lim_{y_1\searrow\mu_1} \mathcal{F}(y_1,y_0)=\lim_{y_1\nearrow\mu_2} \mathcal{F}(y_1,y_0)=+\infty.
$$
Moreover, the inequality $\mathcal{F}(0,y_0)  <  
0$ holds for every $y_0\neq0$ and $\mathcal{F}(0,0)=0$.

Let us consider now $q=0$. For every $y_0\in\mathcal{I}\setminus\{0\}$ there exist two unique and different values 
$\phi_0(y_0),\varphi_0(y_0)\in\mathcal{I}$ such that $\mathcal{F}(\phi_0(y_0),y_0)=\mathcal{F}(\varphi_0(y_0),y_0)=0$. Moreover, the inequality $\phi_0(y_0)\cdot\varphi_0(y_0)<0$ is true (without loss of generality, we can assume $y_0\cdot\phi_0(y_0)>0$). Notice that if $\phi_0$ and $\varphi_0$ 
could be extended to $y_0=0$, both functions should satisfy $\phi_0(0)=\varphi_0(0)=0$ because the only solution to $\mathcal{F}(y_1,0)=0$ is $y_1=0$. 

On the one hand, in view of the trivial equality $\mathcal{F}(y_0,y_0)=0$ for every $y_0\in\mathcal{I}$, 
function
$$
\phi_0:y_0\in\mathcal{I}\longrightarrow\phi_0(y_0)\in\mathcal{I}
$$
must be the identity function and there exist a function $\mathcal{F}^*$ defined in $\mathcal{I}^2$ such that
$\mathcal{F}(y_1,y_0)= (y_1-y_0) \mathcal{F}^* (y_1,y_0)$ and $\mathcal{F}^*(\varphi_0(y_0),y_0)=0$ for every $y_0\in\mathcal{I}\setminus\{0\}$.

The function $\phi_0$ is clearly analytic because it is the identity function. From \eqref{eq:derivadasF}, it is immediate that $\partial \mathcal{F}/\partial y_1(\varphi_0(y_0),y_0)\neq0$ for every $y_0\neq0$ and so the function $\varphi_0$ is analytic for every $y_0\neq0$ as a direct consequence of the implicit function theorem for real analytic functions (see Lemma \ref{lemma:analitica}). 

Now it will be proved that $\varphi_0$ can be analytically extended to $y_0=0$. 
From the expressions of the derivatives given in \eqref{eq:derivadasF}, it is clear that $\partial \mathcal{F}/\partial y_1(0,0)=\partial \mathcal{F}/\partial y_0(0,0)=0$, all the mixed partial derivatives of $\mathcal{F}$ vanish and the equality
$$
\frac{\partial^n \mathcal{F}}{\partial y_1^n}(y_1,y_0)=-\frac{\partial^n \mathcal{F}}{\partial y_0^n}(y_0,y_1)
$$
is satisfied for every positive integer $n$.
Thus, the function $\mathcal{F}^*$ is analytic in a neighborhood of the origin, $\mathcal{F}^*(0,0)=0$, $\partial \mathcal{F}^*/\partial y_1(0,0)=\partial \mathcal{F}^*/\partial y_0(0,0)=a^{-2}\neq0$. 
As a direct consequence of the implicit function theorem applied to $\mathcal{F}^*$ at the origin, it is obtained that $\varphi_0$ can be analytically extended to $y_0=0$ and $\varphi_0(0)=0$.

On the other hand, taking into account that
\begin{equation}
\label{eq:conveniointegral}
\mathcal{F}(y_1,y_0)=-\mathcal{F}(y_0,y_1)
\end{equation}
then the function 
$$
\varphi_0:y_0\in\mathcal{I}\longrightarrow\varphi_0(y_0)\in\mathcal{I}
$$
is an involution.

In the proof of statements \ref{th:it:a=0} and \ref{th:it:aneq0.q=0}, variable $y_1$ has been obtained as a function of $y_0$ in the complete interval $\mathcal{I}$ for the level curves of function $\mathcal{F}$. Note that, analogously, variable $y_0$ could have been obtained as a function of $y_1$ in the complete interval.

The proof of statement \ref{th:it:aneq0.qneq0} (where $a\cdot q\neq0$) is similar to the proof of statement \ref{th:it:aneq0.q=0} (in fact, it is even easier because the level curves do not contain the origin and the singular case does not appear).

Let us consider $q>0$. For every $y_0\in\mathcal{I}$ there exist two unique and different values 
$\phi_q(y_0),\varphi_q(y_0)\in\mathcal{I}$ such that $\mathcal{F}(\phi_q(y_0),y_0)=\mathcal{F}(\varphi_q(y_0),y_0)=q$. Moreover, the inequality $\phi_q(y_0)\cdot\varphi_q(y_0)<0$ is true (without loss of generality, we can assume that function $\varphi_q$ is strictly negative). By means of the implicit function theorem for real analytical functions, functions $\phi_q$ and $\varphi_q$ are analytic in the complete interval $\mathcal{I}$ (see Lema \ref{lemma:analitica}) and satisfy the inequalities 
$$
  y_0 \frac{d\phi_q}{d y_0}(y_0)=\frac{y_0^2 V(0,\phi_q(y_0))}{\phi_q(y_0) V(0,y_0)}>0, \quad y_0 \frac{d\varphi_q}{d y_0}(y_0)=\frac{y_0^2 V(0,\varphi_q(y_0))}{\varphi_q(y_0) V(0,y_0)}<0
$$
for $y_0\in\mathcal{I}\setminus\{0\}$. Therefore, both functions are unimodal and their critical point (minimum for $\phi_q$ and maximum for $\varphi_q$) are located at $y_0=0$. 

For $q<0$, an analogous reasoning with a simple interchange of variables $y_1$, $y_0$ and assuming that $\varphi_q$ is strictly positive, leads to the corresponding result.
Besides that, since \eqref{eq:conveniointegral} holds
then $\phi_q\equiv\varphi_{-q}$ for every $q\in\mathbb{R}\setminus\{0\}$.

Finally, again from the divergence of the integrals shown in \eqref{eq:divergenciaextremos}, it is now obvious to see that the restricted functions given in \eqref{eq:restrictedfunctions} are bijective. 

\end{proof}

\begin{remark}
\label{rem:biyectividadgeneral}
The graphs of the restricted functions given in \eqref{eq:restrictedfunctions} are contained into the fourth quadrant
since this is the natural quadrant for the left Poincar\'e half-map. The restrictions to this fourth quadrant of the functions $\varphi_q$ given in items \ref{th:it:a=0} and \ref{th:it:aneq0.q=0} of Theorem \ref{th:curvasnivel} are also bijective.
 Obviously, the functions $\phi_q$ and $\varphi_q$ can be also restricted to the other quadrants to get bijective functions.

\end{remark}

In Figure \ref{fig:curvasnivel}, several level maps of $\mathcal{F}(y_1,y_0)$ are shown.

\begin{figure}[!h]
    \begin{center}
    \begin{tabular}{c@{}c@{}c}
     \includegraphics[width=0.31\linewidth]{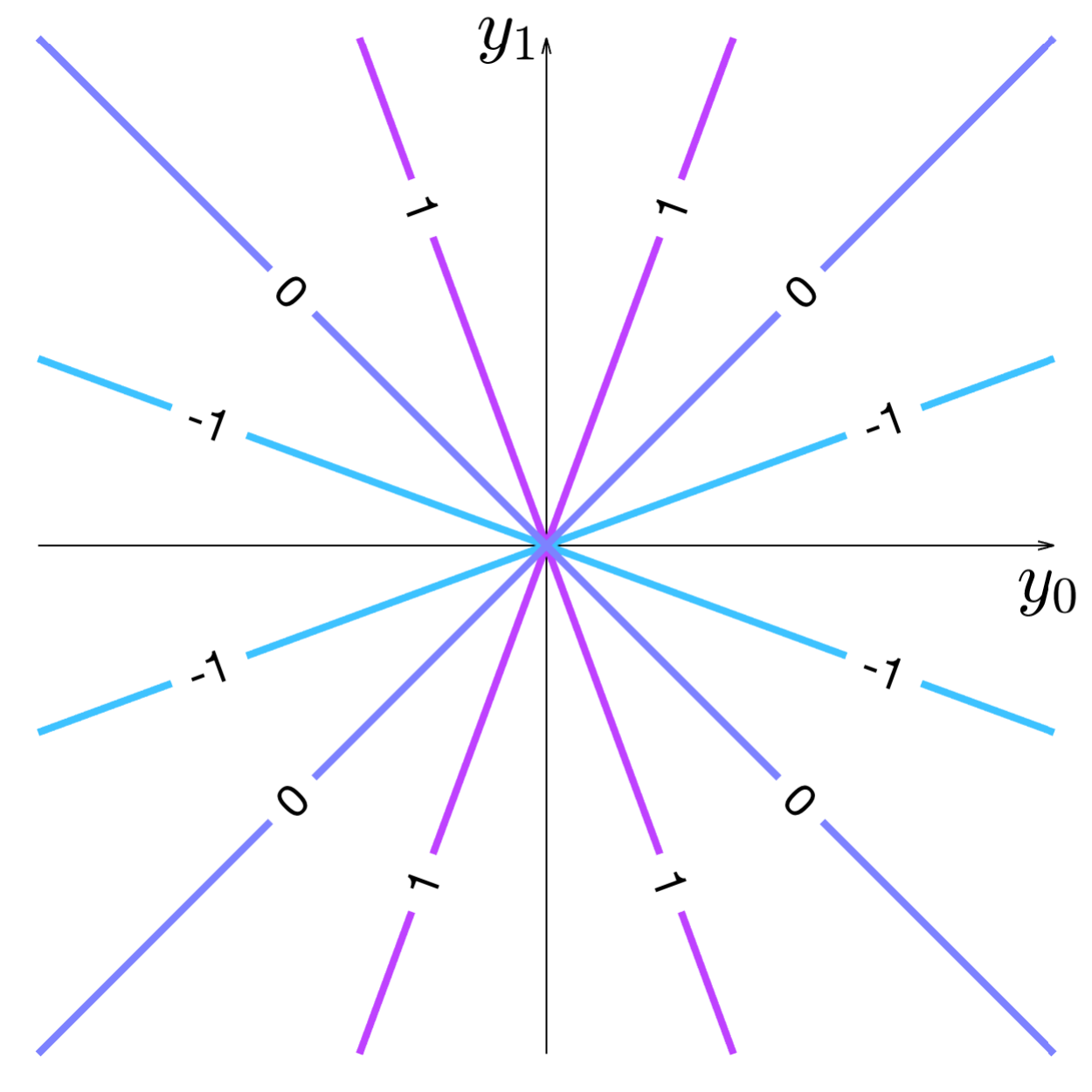}&\ \
     \includegraphics[width=0.31\linewidth]{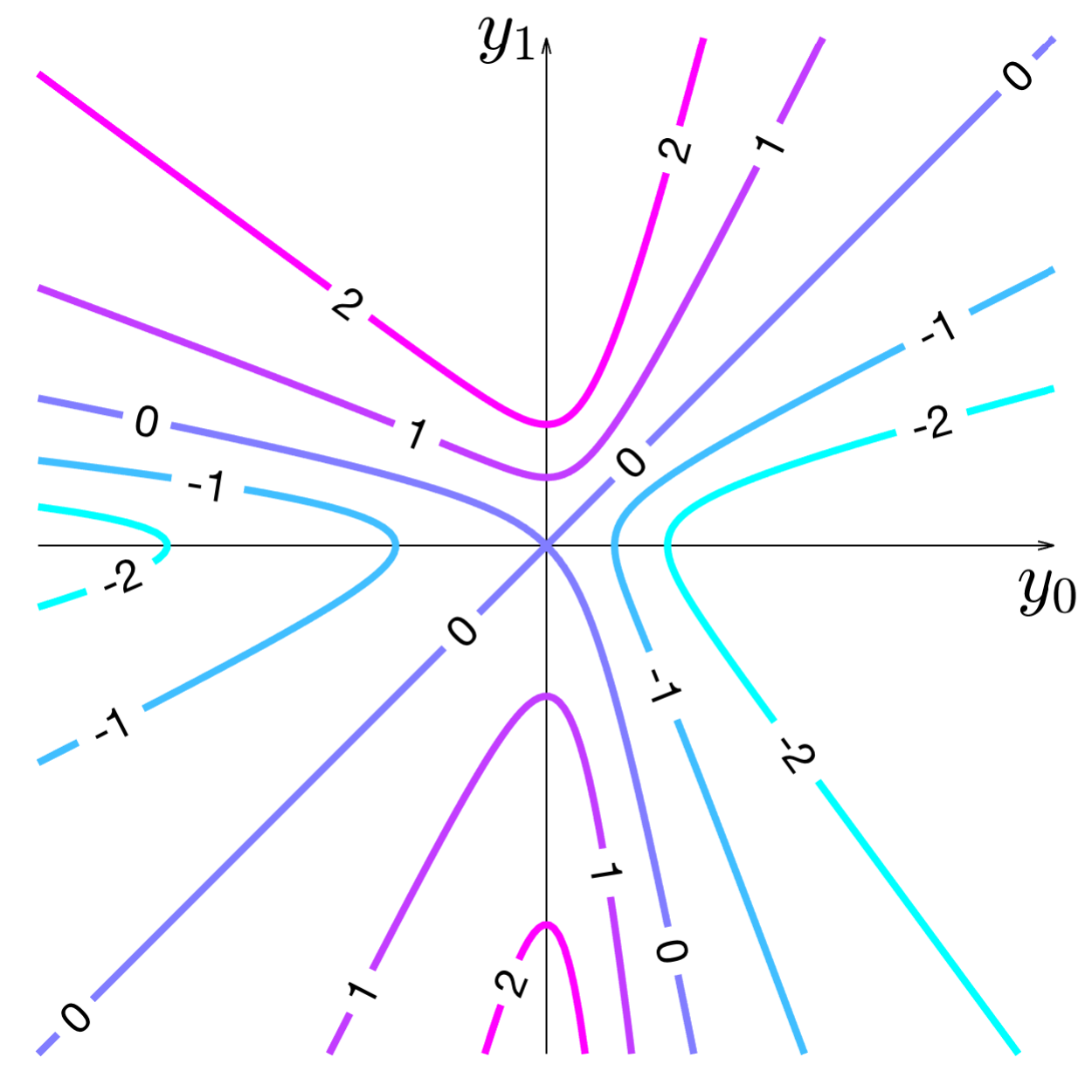}&\ \
     \includegraphics[width=0.31\linewidth]{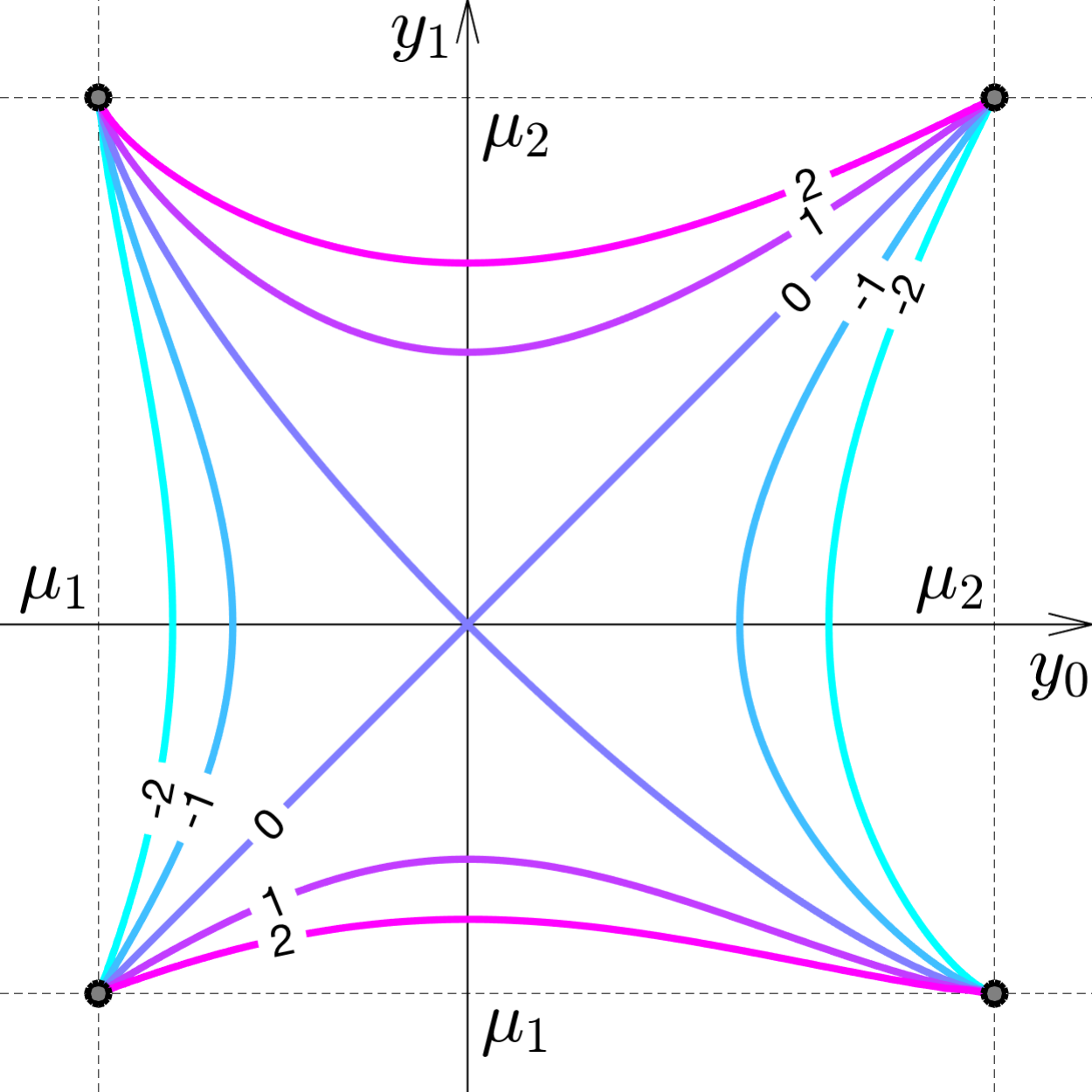}\\
     {\footnotesize (a)} & {\footnotesize (b)}& {\footnotesize (c)}
    \end{tabular}
    \end{center}
     \caption{Level curves of the index-like function $\mathcal{F}$ in $\mathcal{I}^2$ for several values $q\in\mathbb{R}$ when: (a) $a=0$; (b) $a\neq0$ and $4D-T^2>0$; (c) $a\neq0$ and $D<0$. Case (a) illustrates item 1 of Theorem \ref{th:curvasnivel}. Case (b) corresponds to items 2 ($q=0$) and 3 ($q\neq0$) for values of the parameters such that system \eqref{eq:lienard} has a focus or a center. Case (c) corresponds to items 2 ($q=0$) and 3 ($q\neq0$) for values of the parameters such that system \eqref{eq:lienard} has a saddle equilibrium and, therefore, the set $\mathcal{I}^2$ is bounded.
     }\label{fig:curvasnivel}
\end{figure}

\begin{remark}
An immediate conclusion of the last Theorem is that 
every level curve given by $\mathcal{F}(y_1,y_0)=q$, $q\in\mathbb{R}$, in $\mathcal{I}^2$ is an orbit of the following third degree polynomial planar 
system of differential equations
\begin{equation}
\label{eq:sdo}
\left\{
\begin{array}{rcl}
\dot{x}&=&y \, V(0,x),\\
\dot{y}&=&x \, V(0,y).
\end{array}
\right.
\end{equation}
Reciprocally, for $a\neq0$ every orbit of system \eqref{eq:sdo} in $\mathcal{I}^2$ is a level curve of the function $\mathcal{F}$. For $a=0$, this is also true except for the coordinate axes, which are foliated by equilibria of system \eqref{eq:sdo}.
Furthermore, for $a\neq0$, the origin is the unique equilibrium point in $\mathcal{I}^2$ and it is a saddle point (see Figure \ref{fig:curvasnivel}).

A differential equation analogous to \eqref{eq:ode} appeared in \cite{GGG10} as an equation for local Poincar\'e maps, close to an orbit, between two transversal sections to this orbit.  
As it can be deduced from the rest of this work, equation \eqref{eq:ode} characterizes the global left (and right) Poincar\'e half-map of system \eqref{eq:lienard} associated to the section $\Sigma$. Obviously, to identify
the solution corresponding to these maps, a suitable initial condition must be added to the differential equation in both cases.

Anyway, as evidenced in \cite{CaFeGaNoPR,CaFeNoPR1,CaFeNoPR2},
the application to many problems of the integral equation \eqref{eq:curvasnivel}, the differential equation \eqref{eq:ode}, or the system \eqref{eq:sdo}, does not need explicit integration. Therefore, the properties of the Poincar\'e half-maps can be directly obtained from them by avoiding the appearance of unnecessary cases.
\end{remark}

\begin{remark}
For $a\neq 0$ the origin is a saddle point of the surface $z=\mathcal{F}(y_1,y_0)$. From this point of view and roughly speaking,
the set of level curves of the index-like function $\mathcal{F}$ can be understood locally as the universal unfolding of $y_1^2-y_0^2$ (see simple bifurcations at \cite{GoSc85}).
\end{remark}

The analysis of the level curves of the index-like function $\mathcal{F}$ led us to the bijective functions $\varphi_q$ restricted to the fourth quadrant that have been mentioned in Remark \ref{rem:biyectividadgeneral}. For some concrete values of $q$, these functions (if $q\geqslant0$) or their inverse ones (if $q<0$) will give the left Poincar\'e half-map of system
\eqref{eq:lienard}. The following section is devoted to proof this assertion, what closes the integral characterization of the Poincar\'e half-maps.

\section{Recovery of half-Poincar\'e maps and flight time}\label{sec:recovery}
This section is devoted to the formulation and proof of Theorem  \ref{th:implicaizquierda}, a reciprocal to Theorem \ref{th:implicaderecha}. That is, not only a Poincar\'e half-map defines an integral relationship between the initial point and its image (see equation \eqref{eq:curvasnivel}) but this integral relationship defines, for the correct values of the parameters, the Poincar\'e half-map. In order to do so, it is convenient to 
state and prove the following Lemma.

\begin{lemma}
Let us consider the inverse integrating factor $V$
given in \eqref{eq:fii} and the index-like function $\mathcal{F}$ given in \eqref{eq:F}. Let us assume that condition \eqref{eq:cond_aD} holds and 
that the open interval $\mathcal{I}$ defined in \eqref{eq:I} is not empty.
Let be $c\in\mathbb{R}$ and $y_0,y_1\in \mathcal I$
such that the equality
$
\displaystyle 
\mathcal{F}(y_1,y_0)
=c T
$ 
holds.
Then, the equality
\begin{equation}
\label{eq:equivalencialogaritmo}
\log\left(\frac{V(0,y_1)}{V(0,y_0)}\right)= T \left(2 D c
+
\int_{y_1}^{y_0}\frac{a}{V(0,y)}dy\right)
\end{equation}
is true. In fact, for $D\neq0$, both equalities
are equivalent.
\end{lemma} 
\begin{proof}
For $D=0$, since condition \eqref{eq:cond_aD} holds, then $a\neq 0$ and it is direct to see that \eqref{eq:equivalencialogaritmo}
is true for every values $y_0$ and $y_1$ (even if $\mathcal{F}(y_1,y_0)\neq
c T$).

For the generic case $D\cdot a\neq0$, it is trivial that the Cauchy principal value can be obviated 
from the definition of $\mathcal{F}$. 
Since
$$
  \displaystyle \int_{y_1}^{y_0} \frac{-y}{{V(0,y)}} dy =-\frac{1}{2 D}\int_{y_1}^{y_0} \frac{2 D y-aT}{{D y^2-a T y+a^2}} dy
  +\frac{1}{2 D}\int_{y_1}^{y_0} \frac{-aT}{{D y^2-a T y+a^2}} dy,
$$
the proof is direct.

For the case $a=0$, it is $\frac{a}{V(0,y)}\equiv0$ and therefore the second term of the right-hand side of equation 
\eqref{eq:equivalencialogaritmo} is zero. Moreover, from condition  \eqref{eq:cond_aD}, the parameter $D$ does not vanish.
Besides that, from equality \eqref{eq:fparaa=0} it is known that
$$
2Dc T=\displaystyle 2 D\operatorname{PV}\!\!\int_{y_1}^{y_0} \frac{-y}{V(0,y)} dy =
2 \log\left|\frac{y_1}{y_0}\right|=\log\left(\frac{V(0,y_1)}{V(0,y_0)}\right).
$$
This concludes the proof.
\end{proof}

Equation \eqref{eq:equivalencialogaritmo} together with equality \eqref{expr:Vphi} suggest an expression for the left flight time. In fact, this expression (shown in \eqref{eq:tiempodevuelo}) will be a crucial element for the proof of the next Theorem, where we take the last step to definitively show the equivalence between Poincar\'e half-maps and some suitable level curves of the index-like function $\mathcal{F}$ given in \eqref{eq:F}.

\begin{theorem} \label{th:implicaizquierda}
Let us consider the inverse integrating factor $V$
given in \eqref{eq:fii} and the index-like function $\mathcal{F}$ given in \eqref{eq:F}. Let us assume that condition \eqref{eq:cond_aD} holds and 
that the open interval $\mathcal{I}$ defined in \eqref{eq:I} is not empty.
Let be $c\in\mathbb{R}$ and
$y_0,y_1\in\mathcal{I}$, $y_0 \geqslant 0$, $y_1 \leqslant 0$ such that
\begin{equation}
\label{eq:cT}
  \displaystyle 
  \mathcal{F}(y_1,y_0)
=c T.
\end{equation}
Assume that one of the following conditions holds:
\begin{enumerate}
\item[(i)] {$c=0$} and {$a>0$}, \quad
\item[(ii)] {$\displaystyle c=\frac{\pi}{D\sqrt{4D-T^2}}\in\mathbb{R}$} and $a=0$,\quad
\item[(iii)] {$\displaystyle c=\frac{2 \pi}{D\sqrt{4D-T^2}}\in\mathbb{R}$} and $a<0$.
\end{enumerate}
Then $y_1$ is the image of $y_0$ by the left Poincar\'e half-map of system \eqref{eq:lienard} related to 
the Poincar\'e section $\Sigma=\{x=0\}$.
Moreover, the corresponding left flight time is 
\begin{equation}
\label{eq:tiempodevuelo}
  \displaystyle \tau(y_0)=2 D c+\int_{y_1}^{y_0} \frac{a}{{V(0,y)}}dy.
\end{equation}
\end{theorem}
\begin{proof}
The proof is immediate for the limit cases of the left Poincar\'e half-map when $y_0=y_1=0$ and point $(y_0,y_1)=(0,0)$ is an invisible tangency or the equilibrium of system \eqref{eq:lienard} (see Remark \ref{rem:tangenciainvisible}), that is, when $y_0=y_1=0$ and $a\geqslant0$. From now till the end of the proof, we assume that this situation does not occur.

Let us suppose that there exist $y_0,y_1\in\mathcal{I}$, $y_0 \geqslant 0$, $y_1 \leqslant 0$ satisfying equality \eqref{eq:cT}. Let us assume that 
one of the items (i), (ii), or (iii) holds.
 Let us define $\tilde{\tau}$ as the right-hand side of the equality given in \eqref{eq:tiempodevuelo} and let be $\Psi(t;y_0)=(\Psi_1(t;y_0),\Psi_2(t;y_0))$ the orbit of system  \eqref{eq:lienard} defined in \eqref{eq:orbitaflujo}.
Under these assumptions, the theorem will be proved if the following conditions are verified:
\begin{itemize}
\item[(C1)] $\tilde{\tau}>0$.
\item[(C2)] $\Psi_1(t;y_0)<0$ for every $t\in(0,\tilde{\tau})$.
\item[(C3)] $\Psi(\tilde{\tau};y_0)=(0,y_1)$.
\end{itemize}

Condition (C1) is trivial for items (i) and (ii). The proof of condition (C1) for item (iii), where $a$ is negative and $c=\frac{2 \pi}{D\sqrt{4D-T^2}}\in\mathbb{R}$, is a direct conclusion from the inequalities
$$
0\leqslant \int_{y_1}^{y_0} \frac{-a}{{V(0,y)}}dy \leqslant \int_{-\infty}^{+\infty} \frac{-a}{{V(0,y)}}dy=\frac{2 \pi}{\sqrt{4D-T^2}}=Dc.
$$

For the proof of conditions (C2) and (C3) we will assume that the trace $T$ does not vanish. Notice that if the proof of the treorem is obtained for $T\neq0$, the results are immediately extended to the case $T=0$ by using the continuity of solutions of differential equations and integrals with respect to parameters.
After the proof of this theorem, an alternative proof for the case $T=0$ is given in Remark \ref{rem:pruebaalternativaT=0},
where some interesting observations
are also added.

The following reasoning proves that $\Psi_1(\tilde{\tau};y_0)\leqslant0$ and $\Psi_1(t;y_0)<0$ for every $t\in(0,\tilde{\tau})$.

Let us assume that there exists a value $\tau^*\in(0,\tilde{\tau}]$ such that $\Psi_1(\tau^*;y_0)=0$ and $\Psi_1(t;y_0)<0$ for every $t\in(0,\tau^*)$. This fact implies that $y_1^*:=\Psi_2(\tau^*;y_0)\leqslant0$.
From the definition of the left Poincar\'e half-map it is $y_1^*=P(y_0)$ and from
Theorem \ref{th:implicaderecha} it is clear that  $\operatorname{PV}\int_{y_1^*}^{y_0} \frac{-y}{{V(0,y)}} dy =c T$ for the corresponding item (i), (ii) or (iii). The bijectivity of functions $\varphi_{q}$ 
when $y_0\geqslant0$ and $y_1\leqslant0$  given in Theorem \ref{th:curvasnivel} implies that
$y_1^*=y_1$.

Property  \eqref{expr:Vphi}, applied to system \eqref{eq:lienard}, implies
$V(0,y_1^*)=\exp\left({T \tau^*}\right) \, V(0,y_0)$, because $\operatorname{div}L=T$ for vector field $L$ defined in \eqref{eq:lienard_L}.
From relationship \eqref{eq:equivalencialogaritmo} it is $V(0,y_1)=\exp\left({T \tilde{\tau}}\right) \,V(0,y_0) $.
Since $y_1^*=y_1$ and $T\neq0$, the equality $\tilde{\tau}=\tau^*$ is true.

Therefore, $\Psi_1(\tilde{\tau};y_0)\leqslant0$ and $\Psi_1(t;y_0)<0$ for every $t\in(0,\tilde{\tau})$. This means that condition (C2) holds.

Finally, in order to prove condition (C3),  we are going to integrate the orthogonal vector field $G$ defined in \eqref{eq:campoortogonal} on a suitable closed curve $\widetilde{\Gamma}=\widetilde{\Gamma}_1\cup\widetilde{\Gamma}_2\cup\widetilde{\Gamma}_3$, constructed by joining a segment $\widetilde{\Gamma}_1$ contained at the separation line $\Sigma$, a piece of an orbit
$\widetilde{\Gamma}_2$ of the system and a piece of a level curve $\widetilde{\Gamma}_3$ of the inverse integrating factor $V$ (see Figure \ref{fig:pruebatiempovuelo}). Concretely,
the first two curves are the segment $\widetilde{\Gamma}_1=\{(0,y)\in\mathbb{R}^2: y\in(y_1,y_0)\}$ and the piece of orbit $\widetilde{\Gamma}_2=\{\Psi(t;y_0): t\in[0,\tilde{\tau})\}$. 

\begin{figure}[h!]
\begin{center}
\includegraphics[width=0.5\linewidth]{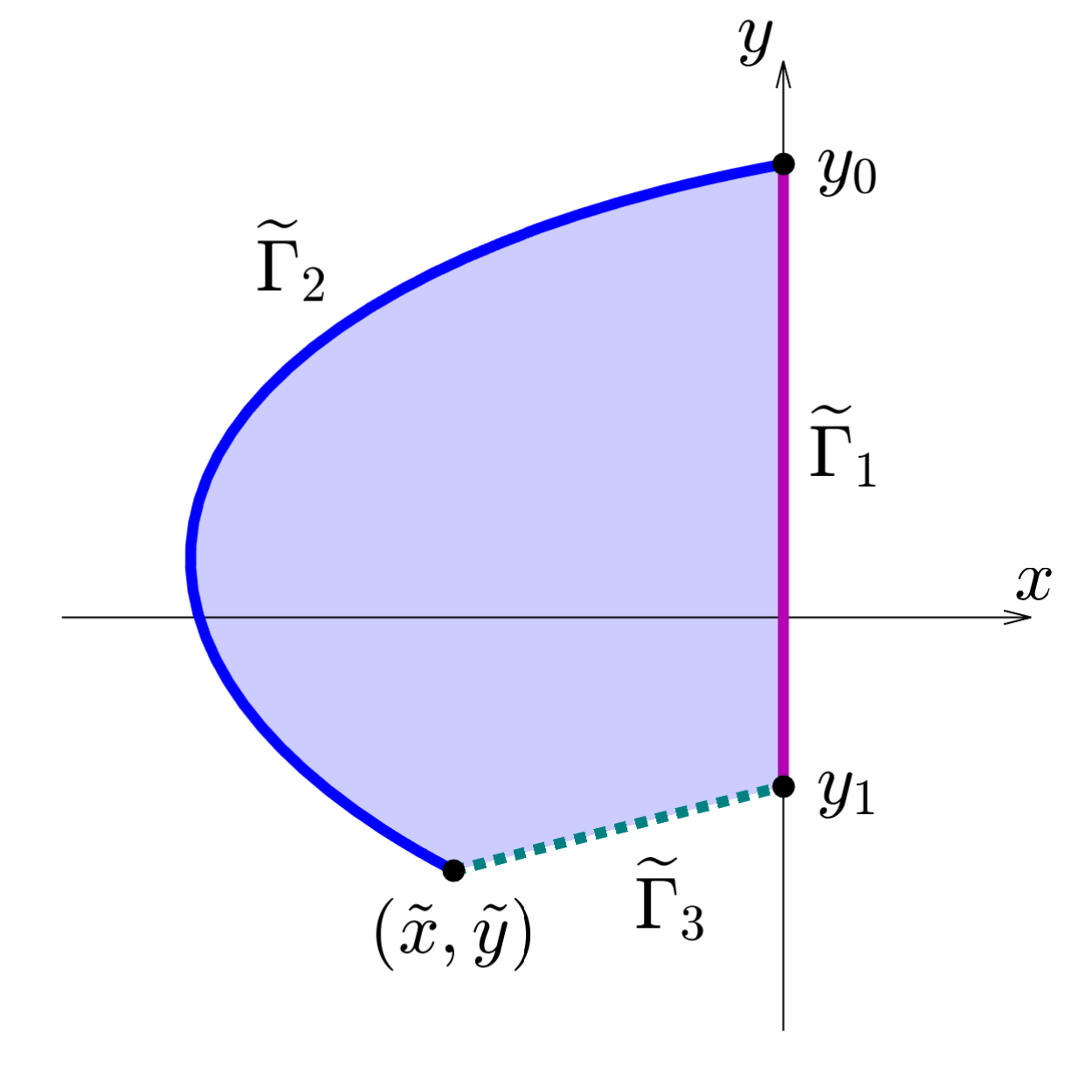}
\end{center}
\caption{Schematic drawing of curve $\widetilde{\Gamma}=\widetilde{\Gamma}_1\cup\widetilde{\Gamma}_2\cup\widetilde{\Gamma}_3$ defined in the proof of Theorem \ref{th:implicaizquierda}. 
Curve $\widetilde{\Gamma}_3$ is drawn as a dotted line because points $(\tilde{x},\tilde{y})$ and $(0,y_1)$ are proved to coincide. }\label{fig:pruebatiempovuelo}
\end{figure}

To define the last portion of $\widetilde{\Gamma}$, let us consider point $(\tilde{x},\tilde{y})=\Psi(\tilde{\tau};y_0)$, that satisfies the equality $V(\tilde{x},\tilde{y})=\exp\left({T \tilde{\tau}}\right) \, V(0,y_0)=V(0,y_1)$ as long as property \eqref{expr:Vphi} and relationship \eqref{eq:equivalencialogaritmo} hold.
Therefore, since $y_0, y_1 \in \mathcal{I}$, points $(\tilde{x},\tilde{y})$ and $(0,y_1)$ are connected by a piece of a level curve of the inverse integrating factor $V$ where, moreover, $V$ is strictly positive. Notice that the portion of the level curve of function $V$ that connects these points is unique except for $4D-T^2>0$, where the level curves are ellipses (see Remark \ref{rem:elipses}), and hence there are two choices of simple curves. Anyway, it is always possible to select this piece of level curve $\widetilde{\Gamma}_3$ so that the complete closed curve $\widetilde{\Gamma}$ encircles no equilibria for case (i) or the equilibrium of system \eqref{eq:lienard} for case (iii). Trivially, in case (ii) the equilibrium is located at segment $\widetilde{\Gamma}_1$.

By reasoning as in the proof of Theorem \ref{th:implicaderecha}, the integration of the orthogonal vector field $G$ along the closed curve
$\widetilde{\Gamma}$ together with equality \eqref{eq:cT} leads to 
\begin{equation}
\label{eq:curvafalsa}
\int_{\widetilde{\Gamma}_3} G \cdot d\mathbf{r}=0.
\end{equation}
On the other hand, since $\left.V\right|_{\widetilde{\Gamma}_3}$ is constant, the parameterization of $\widetilde{\Gamma}_3\equiv \mathbf{r}(s)=(x(s),y(s))$, $s\in[s_0,s_1]$, may be chosen in such a way that its tangent vector is $\nabla V (x(s),y(s))^{\perp}
$, where the orthogonal vector is taken as defined in Remark \ref{rem:ortogonal}. From equalities \eqref{eq:curvafalsa} and \eqref{eq:iif_ortogonal}, it holds
$$
0=\int_{\widetilde{\Gamma}_3} G \cdot d\mathbf{r}=\int_{s_0}^{s_1} G(x(s),y(s)) \cdot \nabla V(x(s),y(s))^{\perp} ds=\int_{s_0}^{s_1} T ds=T (s_1-s_0).
$$
Since $T\neq0$ then $s_1=s_0$. That is, $(\tilde{x},\tilde{y})=(0,y_1)$ and the proof concludes. 
\end{proof}

\begin{remark}\label{rem:pruebaalternativaT=0}
In the proof of Theorem \ref{th:implicaizquierda}, the singular features of case $T=0$ have forced us to analyze it separately. Specifically, since for $T=0$, the inverse integrating factor $V$ is constant along each orbit of system \eqref{eq:lienard}, all the properties based on the variation of $V$ along the orbits are useless. For instance, the expression
$$
\tau=\frac1T\log\left(\frac{V(0,y_1)}{V(0,y_0)}\right),
$$
derived from equality \eqref{expr:Vphi} and used to obtain the left flight time $\tau$ between two points of the Poincar\'e map, can only be used for $T\neq0$.

Besides the presented proof for case $T=0$ (based on arguments of continuity with respect of parameters), there are two alternative proofs based on  ideas we would like to highlight. Notice that the reversibility of system \eqref{eq:lienard} for $T=0$ implies that $y_1=-y_0$.

The first proof (the classical way for linear systems) requires the integration of system \eqref{eq:lienard} for $T=0$ (trivial integration but with several distinguished cases), the obtention of the left flight time by means of the imposition of the left Poincar\'e half-map conditions and the verification of the equality between the computed left flight time and the one given in \eqref{eq:tiempodevuelo}.

For the second alternative proof for $T=0$ (based on the computation of the left flight time from the hamiltonian character of the system),
we use the conservation of $V(x,y)=(Dx-a)^2+D y^2$ along the orbits of system \eqref{eq:lienard}, that is, at any point $(x,y)$ of the orbit that passes through point $(0,y_0)$ it holds $\dot{y}^2+D y^2=(Dx-a)^2+D y^2=a^2+D y_0^2$. For case (i) of Theorem \ref{th:implicaizquierda} and the piece of orbit contained in $\left\{x<0\right\}$ that connect points $(0,\pm y_0)$, it is true that $\dot{y}<0$ and so $\dot{y}=-\sqrt{a^2+D y_0^2-D y^2}$. Therefore, the left flight time is
\begin{equation} \label{eq:tiempovueloalternativo}
\tau(y_0)=\int_{y_0}^{-y_0} \frac{-dy}{\sqrt{a^2+D y_0^2-D y^2}}.
\end{equation}
It is direct to check that the change of variable $$w=\sqrt{\frac{V(0,y_0)}{V(0,y)}}\, y,$$
transforms the integral of expression \eqref{eq:tiempodevuelo} into \eqref{eq:tiempovueloalternativo}.

In cases (ii) and (iii), the condition $T=0$ implies that system \eqref{eq:lienard} corresponds to a linear center and all the orbits are periodic with period $\frac{4 \pi}{\sqrt{4D-T^2}}$. Thus, for case (ii) the proof is immediate (the orbit is half of a complete periodic orbit) and for case (iii) expression \eqref{eq:tiempodevuelo} must be understood as the complete period minus the right flight time.
\end{remark}

The immediate and more relevant conclusion of Theorems \ref{th:implicaderecha} and \ref{th:implicaizquierda} is the following Corollary.

\begin{corollary}
\label{cor:final}
Let us assume that condition \eqref{eq:cond_aD} holds and that the open interval $\mathcal{I}=(\mu_1,\mu_2)$ defined in \eqref{eq:I} is not empty.  Let be functions $\varphi_q$, $q\in\mathbb{R}$ as given in Theorem \ref{th:curvasnivel}. 
Assume that $c\in\mathbb{R}$ satisfies one of the conditions (i), (ii), (iii) of Theorem \ref{th:implicaizquierda}.

If the trace verifies $T\geqslant0$ then the left Poincar\'e half-map of system \eqref{eq:lienard} related to 
the Poincar\'e section $\Sigma$ is $\varphi_{cT}:[0,\mu_2) \longrightarrow (\mu_1,\varphi_{cT}(0)]$.

If the trace verifies $T<0$ then the left Poincar\'e half-map of system \eqref{eq:lienard} related to 
the Poincar\'e section $\Sigma$ is $\varphi^{-1}_{cT}:[\varphi_{cT}(0),\mu_2) \longrightarrow (\mu_1,0]$.
\end{corollary}

\begin{remark}
For $T=0$, the functions $\varphi_{0}$ and $\varphi^{-1}_{0}$ coincide (see Theorem \ref{th:curvasnivel}). Therefore,
the case $T=0$ could have been also joined to the case $T<0$ in the statement of Corollary \ref{cor:final}.
\end{remark}

\begin{remark}
Let us briefly explain the role of the value $\varphi_{cT}(0)$ that acts as a limit
point of the range or domain of  the left Poincar\'e half-map $\varphi_{cT}$ or $\varphi^{-1}_{cT}$ respectively. For $cT=0$,
it is $\varphi_{0}(0)=0$ (see Theorem \ref{th:curvasnivel}) and it obviously
corresponds to the tangency point for scenarios ($S_0$), ($S_1$) and the center case of ($S_2$) (see Figures \ref{fig:casoS0}, \ref{fig:casoS1} and 
\ref{fig:casoS2}(a)).  When $cT\neq0$ and $T>0$, then $\varphi_{cT}(0)=\hat{y}_1<0$ (see Theorem \ref{th:curvasnivel} and Figure
\ref{fig:casoS2}(c) corresponding to scenario ($S_2$)), in other words, it is the image of $y_0=0$ by means of the left Poincar\'e half-map.
When $cT\neq0$ and $T<0$, then $\varphi_{cT}(0)=\hat{y}_0>0$ (see Theorem \ref{th:curvasnivel} and Figure
\ref{fig:casoS2}(b) corresponding to scenario ($S_2$)), in other words, it is the pre-image of $y_0=0$ by means of the left Poincar\'e half-map.
\end{remark}

Regarding the functions given in Corollary \ref{cor:final}, another important and direct consequence of Theorems \ref{th:implicaderecha}, \ref{th:curvasnivel}, and \ref{th:implicaizquierda},  characterizes the analyticity of the left Poincar\'e half-map. 

As it is well-known, for a point $y_0\geqslant0$ with $P(y_0)<0$, the analyticity of $P$ at $y_0$ is ensured by the transversality between the flow and the separation line (see,  for instance, \cite{Chi92}). The tangency of the flow of system \eqref{eq:lienard} at the origin complicates the study (see some partial results at \cite{ColGasPro01}). However, from Theorem \ref{th:curvasnivel}, the conclusion is immediate.

\begin{corollary}
\label{cor:analiticidad} Let be $P$ the left Poincar\'e half-map of system \eqref{eq:lienard}. The following statements are true:
\begin{enumerate}
\item If $P(0)=0$ then $P$ is an involution and, moreover, it is analytical in its domain of definition.
\item If $P(0)<0$ then $P$ is analytical in its domain of definition and $P^{-1}$ is analytical in the interior of its domain of definition.
\item If $P^{-1}(0)>0$ then $P$ is analytical in the interior of its domain of definition and $P^{-1}$ is analytical in its domain of definition.
\end{enumerate}
\end{corollary}

\begin{remark} 
Finally, Theorem \ref{th:implicaizquierda} and Corollaries \ref{cor:final} and \ref{cor:analiticidad}, can be easily extended to the right Poincar\'e half-map.
\end{remark}

\section{Conclusions}\label{sec:conclusions}
As was said in the introduction, in order to avoid the flaws due to the computation of the solutions of linear systems in the analysis of Poincar\'e half-maps
a novel theory has been developed in this manuscript. The key point was the introduction and study of the index-like function $\mathcal{F}$ given in \eqref{eq:F} that was obtained from the line integration, on a suitable curve, of an orthogonal vector field written in terms of a good choice of inverse integrating factor. In fact, this index-like function gives a common way to express the Poincar\'e half-maps.

This new approach could be extended to the study of Poincar\'e half-maps for non-linear planar systems as far as a nice inverse integrating factor may be found. Another interesting extension of the theory could be an analogous analysis for higher dimensions, where the important role of inverse integrating factor should be assumed by inverse Jacobi multipliers (see \cite{BeGia03}). We are convinced that these two ideas will open fruitful lines of study in the short or medium term.  

However, the true importance of the technique developed in this work is currently revealed in its application to the analysis of the dynamical behavior of planar piecewise linear systems, in particular, the obtention of optimal upper bounds on the number of limit cycles. On the one hand, for continuous planar piecewise linear systems with two zones of linearity, it is known that this upper bound is one. This result was originally proved in \cite{FrPoRoTo98} with exhaustive and long case-by-case analysis. By using the index-like function, we have got a direct and short proof (without cases) of this same result (see \cite{CaFeNoPR1}). On the other hand, with the same technique, it is possible to obtain the same bound for sewing discontinuous planar piecewise linear systems with two zones of linearity (see \cite{CaFeNoPR2}). Previous works (see \cite{FrPoTo13,MeTo15,LiLiuLli21}) give partial results for such systems by using the case-by-case analysis. 

It is also interesting to study the optimal upper bound for generic discontinuous planar piecewise linear systems with two zones of linearity. The first basic open problem is the existence of such a uniform bound for all these systems, that is, independent of the value of parameters. Since the use of our new approach allows us to write this problem in terms of the existence of a uniform bound for the number of solutions of a common system of polynomial equations of fixed degrees, the conclusion is obvious. 
Moreover, the optimal bound, which is known to be greater or equal than three \cite{HuYa12,LlPo12,BuPeTo13}, could be directly established (without a case-by-case study) from this system of polynomial equations. Nowadays, this study is one of our priority lines of research.

Regarding other interesting achieved results in this manuscript, we would like to mention the analyticity of the Poincar\'e half-map or its inverse function at the tangency points.

\section*{Acknowledgments}
The authors would like to express their gratitude to professors Douglas D. Novaes and Jos\'e A. Rodr\'{\i}guez for valuable and constructive discussions and priceless encouragement during part of the development and writing of this work.

This work has been partially supported by the \emph{Ministerio de Econom\'ia y Competitividad} co-financed with FEDER funds, in the frame of the projects MTM2014-56272-C2-1-P, MTM2015-65608-P, MTM2017-87915-C2-1-P and PGC2018-096265-B-I00 and by the \emph{Consejer\'{\i}a de Educaci\'on y Ciencia de la Junta de Andaluc\'{\i}a} (TIC-0130, P12-FQM-1658).

\end{document}